\documentclass[10pt,twoside,reqno]{amsart}
\usepackage{lmodern}
\usepackage[T1]{fontenc}
\usepackage{amsmath,amsfonts,amsthm,amssymb,amscd}
\usepackage{mathrsfs}
\usepackage{latexsym}
\usepackage{graphicx}
\usepackage{enumitem}

\usepackage{geometry}
\newtheorem{corollary}{Corollary}[section]
\newtheorem{lemma}{Lemma}[section]
\newtheorem{theorem}{Theorem}[section]

\newtheorem*{remark}{Remark}

\usepackage{hyperref}  
\hypersetup{colorlinks=true, linktoc=all,
    linkcolor=blue,
		citecolor=blue,
		urlcolor=cyan,
		bookmarksopen=true
		}
\newcommand{\lnc}{\mathscr{L}}

\title[Short Weyl sums and Waring's problem with Almost Proportional Summands]
{Intermediate-Range Estimates for Short Weyl Sums and Waring's Problem with Almost Proportional Summands}
\author{Karimjon Ibrohimjonovich Mirzoabdughafurov}
\address{Tajik State University of Finance and Economics, Dushanbe, Tajikistan}
\email{karimjon2003@mail.ru}
\date{}


\begin{document}

\begin{abstract}
We obtain a uniform pointwise estimate for short Weyl sums of degree
$n\ge 3$ in the intermediate rational-approximation range
$$
\frac{1}{qx^{n-2}y}\ll|\lambda|\ll\frac{1}{qy^{n-1}}.
$$
This range arises from a second application of Dirichlet's rational
approximation theorem in estimating the residual integral that occurs
in the derivation of an asymptotic formula for Waring's problem with
almost proportional summands. 

For $r=r(n)=1+\min\{2^n,n(n+1)\}$
and fixed positive numbers $\mu_1,\ldots,\mu_r$ satisfying
$$
\mu_1+\cdots+\mu_r=1,
$$
we derive an asymptotic formula for the number of representations
$$
x_1^n+\cdots+x_r^n=N,\qquad
|x_i^n-\mu_iN|\le H,\qquad 1\le i\le r,
$$
valid for
$$
N^{1-\theta(n,r)+\varepsilon}\le H\le\frac{N}{\ln N},
\qquad
\theta(n,r)=
\begin{cases}
\frac{1}{27}, & n=3,\\
\frac{1}{100}, & n=4,\\
\frac{1}{2n}, & n\ge 5.
\end{cases}
$$
For $n=3,4$, we use the intermediate-range estimate. For $n\ge 5$,
we use a bound obtained by Poisson summation and an estimate derived
from successive rational approximations and the modern form of
Vinogradov's mean value theorem.

The resulting admissible lower bound for $H$ improves the previously
known bound for every $n\ge 3$. For $n\ge 5$, we reduce the number
of summands from $2^n+1$ to $n(n+1)+1$ while substantially lowering
the threshold for $H$ to $N^{1-\frac{1}{2n}+\varepsilon}$.
\end{abstract}
\maketitle

\section{Introduction}

Let \(n\ge 3\), and let \(\mu_1,\ldots,\mu_r\) be fixed positive
numbers satisfying
\[
\mu_1+\cdots+\mu_r=1.
\]
We consider the number \(J_{n,r}(N,H)\) of solutions of the equation
\begin{equation}\label{Formula-x1n+...+xrn=N}
x_1^n+\cdots+x_r^n=N
\end{equation}
in natural numbers \(x_1,\ldots,x_r\), subject to
\[
|x_i^n-\mu_iN|\le H,
\qquad 1\le i\le r.
\]
Thus, the \(n\)-th power of each variable is restricted to a short
interval centred at \(\mu_iN\).
 In the special case  $\mu_1=\cdots=\mu_r=\frac{1}{r},$
this problem reduces to Waring's problem with almost equal summands.

Waring's problem with proportionality conditions was initiated by
E.~M.~Wright~\cite{Wright-1933,Wright-1934}. For
$r\ge(n-2)2^{n-1}+5$,
he obtained an asymptotic formula under the condition
$H=N^{1-\theta_w(n,r)}$,
where the exponent \(\theta_w(n,r)\) is given explicitly in his theorem.
For \(n=3,4,5\), the corresponding pairs \((r,\theta_w(n,r))\) are
\[
\left(9,\frac{1}{51}\right),\qquad
\left(21,\frac{1}{100}\right),\qquad
\left(53,\frac{1}{325}\right).
\]
Wright's work provided the starting point for subsequent studies of
localized versions of Waring's problem, including the case of almost
equal summands.

Daemen subsequently studied localized solutions in Waring's problem
with almost equal summands. Using Vinogradov's mean value theorem and
the binomial descent procedure, he obtained an asymptotic formula for
\(r\ge r_n\), where
\[
r_2=9,\quad r_3=19,\quad r_4=49,\quad r_5=113,\quad r_6=243.
\]
His result concerns variables satisfying
\[
X-Y< x_j\le X+Y,
\qquad
X=\left[\left(\frac{N}{r}\right)^{\frac{1}{n}}\right],
\qquad
Y=X^{\frac{1}{2}}Y_n,
\]
where \(Y_n=(\ln X)^{r_{n-1}}\). Thus, Daemen's asymptotic formula
reaches the square-root scale, up to a logarithmic factor, but requires
a substantially larger number of summands: \(19\) for cubes and \(49\)
for fourth powers. In a related paper, he established a lower bound at
the square-root scale \(Y=cX^{\frac{1}{2}}\)
\cite{Daemen-2010-1,Daemen-2010-2}. He also considered the special case
of thirteen almost equal cubes by an elementary method
\cite{Daemen-2010-3}.

The modern theory of Vinogradov's mean value theorem was developed by Bourgain, Demeter and Guth for degrees greater 
than three and by Wooley for arbitrary degrees~\cite{BDG-2016,Wooley-2019}. Their results yield sharp mean-value estimates 
for polynomial Weyl sums and have important consequences for the classical Waring problem. In a related direction, 
Brüdern and Wooley obtained mean-value estimates for smooth Weyl sums taken over a long range, with the 
integration restricted to sets of minor arcs of intermediate and large height, and applied these estimates 
to Waring's problem for larger powers~\cite{Brudern-Wooley-2023}. 

More recently, Waring's problem with almost proportional summands was
studied in \cite{Rakh-2023,Rakh-2024}. 
A central analytic difficulty in problems of this type is the estimation of short Weyl sums.

We use the notation
\[
T(\alpha;x,y)
=
\sum_{x-y<m\le x}e(\alpha m^n),
\qquad
e(z)=e^{2\pi iz}.
\]

In the classical Waring problem, the corresponding Weyl sums have
length comparable to \(x\). By contrast, in the problem with almost
proportional summands considered here, the variables are restricted
to short intervals, so that \(y\) may be considerably smaller than
\(x\). Consequently, the application of the Hardy--Littlewood circle
method requires estimates for \(T(\alpha;x,y)\) that are uniform both
in \(y\) and in the rational approximation of \(\alpha\).

Vaughan~\cite{Vaughan-1981} obtained approximations for ordinary Weyl
sums by combining van der Corput's method with estimates for complete
exponential sums. An estimate for short cubic Weyl sums was 
established in \cite{Rakh-Mirzo-2008-1} and was applied to Waring's
problem for cubes with almost equal summands in
\cite{Rakh-Mirzo-2008-2}. A fourth-degree analogue was obtained in
\cite{Rakh-Mirzo-2010} and applied to Waring's problem for fourth
powers with almost equal summands in \cite{Rakh-2011}. Higher-degree
extensions and their applications were developed in
\cite{Rakh-2015-1,Rakh-2015-2}. These estimates were subsequently used
in the study of Waring's problem with almost proportional summands in
\cite{Rakh-2023,Rakh-2024}. 

A treatment of the intermediate range for short quartic Weyl sums was
previously developed in \cite{Mirzo-2026} and applied there to
Waring's problem for seventeen fourth powers with almost equal
summands. The present paper extends this approach to short Weyl sums
of arbitrary degree \(n\ge3\) and applies the resulting estimate to
Waring's problem with almost proportional summands for $n=3,4$.

To describe the range covered by these results, write
\[
\alpha=\frac{a}{q}+\lambda,
\qquad
(a,q)=1,
\qquad
1\le q\le\tau,
\qquad
|\lambda|\le\frac{1}{q\tau},
\]
and put
\[
S(a,q)=
\sum_{m=1}^{q}
e\left(\frac{am^n}{q}\right).
\]
We also use the notation
\[
\mathfrak S(N)
=
\sum_{q=1}^{\infty}
\sum_{\substack{0\le a<q\\(a,q)=1}}
\frac{S^r(a,q)}{q^r}
e\left(-\frac{aN}{q}\right).
\]

It was shown in
\cite{Rakh-2015-1} that, under the
condition
\[
\tau\ge 2n(n-1)x^{n-2}y,
\]
the central part of a major arc, defined by
\[
|\lambda|\le\frac{1}{2nqx^{n-1}},
\]
admits the approximation
\[
T(\alpha;x,y)
=
\frac{y}{q}S(a,q)\gamma(\lambda;x,y)
+
O\left(q^{\frac{1}{2}+\varepsilon}\right),
\]
where
\[
\gamma(\lambda;x,y)
=
\int\limits_{-\frac{1}{2}}^{\frac{1}{2}}
e\left(
\lambda
\left(x-\frac{y}{2}+yt\right)^n
\right)\,dt.
\]
In the remaining part of the major arc,
\[
\frac{1}{2nqx^{n-1}}
<
|\lambda|
\le
\frac{1}{q\tau},
\]
one has
\[
|T(\alpha;x,y)|
\ll
q^{1-\frac{1}{n}}\ln q
+
\min\left(
yq^{-\frac{1}{n}},
x^{\frac{1}{2}}q^{\frac{1}{2}-\frac{1}{n}}
\right).
\]

Since
\[
\tau\ge 2n(n-1)x^{n-2}y,
\]
the upper endpoint of this range satisfies
\[
\frac{1}{q\tau}
\ll
\frac{1}{qx^{n-2}y}.
\]
Thus, the estimates quoted above concern the first rational
approximation used in the initial major-arc decomposition and reach
only the range
\[
|\lambda|
\ll
\frac{1}{qx^{n-2}y}.
\]

For $n=3,4$, a different range arises in the treatment of the
residual integral.
On the residual set, Dirichlet's approximation theorem is applied a
second time, giving
\[
\alpha=\frac{a_1}{q_1}+\beta,
\qquad
(a_1,q_1)=1,
\qquad
1\le q_1\le y^{n-1},
\qquad
|\beta|\le\frac{1}{q_1y^{n-1}}.
\]
For the part corresponding to small values of \(q_1\), the definition
of the residual set implies
\[
|\beta|
\gg
\frac{1}{q_1x^{n-2}y}.
\]
Thus, the second rational approximation leads to the intermediate range
\[
\frac{1}{q_1x^{n-2}y}
\ll
|\beta|
\ll
\frac{1}{q_1y^{n-1}}.
\]
After relabelling \(q_1\) and \(\beta\) as \(q\) and \(\lambda\),
respectively, this becomes
\[
\frac{1}{qx^{n-2}y}
\ll
|\lambda|
\ll
\frac{1}{qy^{n-1}}.
\]
This intermediate range is not covered by the estimates quoted above.
Theorem~\ref{ShortExposumIR} provides a uniform estimate for the short
Weyl sum \(T(\alpha;x,y)\) throughout this range.

We now turn to Waring's problem with almost proportional summands and
compare our result with the previously known asymptotic formulas.
In \cite{Rakh-2023}, the case $n=3$, $r=9$ was treated under the condition
$H\ge N^{1-\frac{1}{30}+\varepsilon}$. More generally, \cite{Rakh-2024}
established an asymptotic formula for
\[
r=2^n+1,\qquad H\ge N^{1-\theta_0(n,r)+\varepsilon},\qquad
\theta_0(n,r)=\frac{2}{n(r+1)(n-1)}.
\]
Theorem~\ref{TeorAsForWaringPPSl} below gives a unified result with
$r=r(n)=1+\min\{2^n,n(n+1)\}$ and exponent $\theta(n,r)$.
For $n=3,4$, the number of summands remains $9$ and $17$, respectively,
whereas for $n\ge5$ both the number of summands and the admissible lower
bound for $H$ are reduced. For $3\le n\le10$, the improvements in the exponent are
\[
\begin{aligned}
\theta_0(3,9)&=\frac{1}{30} &\quad\longrightarrow\quad \theta(3,9)&=\frac{1}{27},\\
\theta_0(4,17)&=\frac{1}{108} &\quad\longrightarrow\quad \theta(4,17)&=\frac{1}{100},\\
\theta_0(5,33)&=\frac{1}{340} &\quad\longrightarrow\quad \theta(5,31)&=\frac{1}{10},\\
\theta_0(6,65)&=\frac{1}{990} &\quad\longrightarrow\quad \theta(6,43)&=\frac{1}{12},\\
\theta_0(7,129)&=\frac{1}{2730} &\quad\longrightarrow\quad \theta(7,57)&=\frac{1}{14},\\
\theta_0(8,257)&=\frac{1}{7224} &\quad\longrightarrow\quad \theta(8,73)&=\frac{1}{16},\\
\theta_0(9,513)&=\frac{1}{18504} &\quad\longrightarrow\quad \theta(9,91)&=\frac{1}{18},\\
\theta_0(10,1025)&=\frac{1}{46170} &\quad\longrightarrow\quad \theta(10,111)&=\frac{1}{20}.
\end{aligned}
\]
For $n\ge5$, these comparisons also include a reduction in the number of
summands from $2^n+1$ to $r(n)=n(n+1)+1$, with $31$ fifth powers in place of $33$.

For $n=3,4$, the proof combines the intermediate-range estimate of
Theorem~\ref{ShortExposumIR} with a second application of Dirichlet's
rational approximation theorem. We split the residual arcs as
$\mathfrak m=\mathfrak m_1\cup\mathfrak m_2$, apply the classical Weyl
estimate on $\mathfrak m_1$, and use the intermediate-range estimate on
$\mathfrak m_2$. The major-arc decomposition and the estimates for
$J(\mathfrak M_2)$ in \cite{Rakh-2024} remain applicable under the weaker
restriction on $H$. In evaluating $J(\mathfrak M_1)$, we retain the
quadratic error arising from the expansion of the singular integral.

For $n\ge5$, we retain the decomposition into central, intermediate,
and residual arcs, with parameters adapted to the square-root scale.
The intermediate arcs are treated by estimate~\eqref{ExpsumIR05},
obtained in the proof of Theorem~\ref{ShortExposumIR} before imposing
the bounds on $\lambda$. On the residual arcs,
Lemma~\ref{Lemma-large-denominators} follows from
Lemma~\ref{Lemma-Weyl-VMVT} by successive rational approximations
to the shifted coefficients. Together with Lemma~\ref{Lemma-moment-VMVT},
obtained from Daemen's reduction and the modern mean value theorem,
these bounds give
$r(n)=n(n+1)+1$ and $\theta(n,r)=\frac{1}{2n}$.

\begin{theorem}[Intermediate-range estimate for short Weyl sums]
\label{ShortExposumIR}
Let \(n\ge 3\) be fixed, let
\[
x\ge x_0>0,
\qquad
0<y\le 0.01x,
\]
and suppose that
\[
\alpha=\frac{a}{q}+\lambda,
\qquad
a\in\mathbb Z,
\qquad
q\in\mathbb N,
\qquad
(a,q)=1.
\]
If
\[
\frac{1}{qx^{n-2}y}
\ll
|\lambda|
\ll
\frac{1}{qy^{n-1}},
\]
then
\[
\left|
T\left(\frac{a}{q}+\lambda;x,y\right)
\right|
\ll
q^{1-\frac{1}{n}}\ln q
+
q^{1-\frac{1}{n}}
\ln\left(2+|\lambda|x^{n-1}\right)
+
q^{\frac{1}{2}-\frac{1}{n}}
x^{\frac{n-2}{2}}
y^{\frac{3-n}{2}}.
\]
The implied constant depends at most on \(n\) and on the constants
implicit in the conditions on \(\lambda\).
\end{theorem}

We now state the resulting asymptotic formula for Waring's problem
with almost proportional summands.
\begin{theorem}[Waring's problem with almost proportional summands]
\label{TeorAsForWaringPPSl}
Let \(N\) be a sufficiently large natural number, let \(n\ge 3\) be
fixed, and put
\[
r=r(n)=1+\min\{2^n,n(n+1)\},
\qquad
\theta(n,r)=
\begin{cases}
\frac{1}{27},&n=3,\\[3pt]
\frac{1}{100},&n=4,\\[3pt]
\frac{1}{2n},&n\ge5.
\end{cases}
\]
Let $\varepsilon>0$ be sufficiently small and $\mu_1,\ldots,\mu_r$
be fixed positive numbers satisfying $\mu_1+\cdots+\mu_r=1$.
Denote by \(J_{n,r}(N,H)\) the number of solutions of
\eqref{Formula-x1n+...+xrn=N} in natural numbers
\(x_1,\ldots,x_r\), subject to
\begin{equation}\label{formula |xin-N/n|<=H}
|x_i^n-\mu_iN|\le H,
\qquad
1\le i\le r.
\end{equation}
Suppose that $N^{1-\theta(n,r)+\varepsilon}\le H \le \frac{N}{\lnc}.$
Then
\begin{equation}\label{asymptotic-almost-proportional}
J_{n,r}(N,H)=\frac{2^r\gamma(n,r)}{n^r}\prod\limits_{i=1}^r\mu_i^{-1+\frac{1}{n}}\mathfrak S(N)\frac{H^{r-1}}{N^{r-\frac{r}{n}}}+O\left(\frac{H^{r-1}}{N^{r-\frac{r}{n}}}\left(\frac{1}{\lnc^{r-1}}+\frac{H^2}{N^2}\right)\right),
\end{equation}
where $\lnc=\ln N,$
the singular series \(\mathfrak S(N)\) defined above is bounded below
by a positive constant, and 
\[
\gamma(n,r)
=
\frac{1}{2^r(r-1)!}
\sum_{j=0}^{\frac{r-1}{2}}
(-1)^j
\binom{r}{j}
(r-2j)^{r-1}.
\]
The implied constant in \eqref{asymptotic-almost-proportional} may
depend on \(n\), \(\varepsilon\), and
\(\mu_1,\ldots,\mu_r\).
\end{theorem}

The first values of $r(n)$ and $\theta(n,r)$ in
Theorem~\ref{TeorAsForWaringPPSl} are recorded in Table~\ref{table-theta-values}.

\begin{table}[ht]
\centering
\caption{The values of $r(n)$ and $\theta(n,r)$}
\label{table-theta-values}
\vspace{-8pt}
\begin{tabular}{|c|c|c|c|c|c|c|c|c|}
\hline
$n$ & $3$ & $4$ & $5$ & $6$ & $7$ & $8$ & $9$ & $10$ \\
\hline
$r(n)$ & $9$ & $17$ & $31$ & $43$ & $57$ & $73$ & $91$ & $111$ \\
\hline
$\theta(n,r)$ & $\frac{1}{27}$ & $\frac{1}{100}$ & $\frac{1}{10}$ & $\frac{1}{12}$ & $\frac{1}{14}$ & $\frac{1}{16}$ & $\frac{1}{18}$ & $\frac{1}{20}$ \\
\hline
\end{tabular}
\end{table}

\begin{remark}
If, in addition, $H\le\frac{N}{\lnc^{\frac{r-1}{2}}}$, then
$\frac{H^2}{N^2}\le\frac{1}{\lnc^{r-1}}$, and the error term
in \eqref{asymptotic-almost-proportional} reduces to
\[
O\!\left(\frac{H^{r-1}}{N^{r-\frac{r}{n}}\lnc^{r-1}}\right).
\]
\end{remark}

Taking $\mu_1=\cdots=\mu_r=\frac{1}{r}$ in
Theorem~\ref{TeorAsForWaringPPSl}, we obtain the corresponding
result for almost equal summands.

\begin{corollary}
Let \(N\) be a sufficiently large natural number, and let
$n\ge3$, with $r=r(n)$ and $\theta(n,r)$ as defined in
Theorem~\ref{TeorAsForWaringPPSl}.
Let \(\varepsilon>0\) be sufficiently small. Denote by \(J_{n,r}(N,H)\) the number of solutions of
\eqref{Formula-x1n+...+xrn=N} in natural numbers
\(x_1,\ldots,x_r\), subject to
\[
\left|x_i^n-\frac{N}{r}\right|\le H,
\qquad
1\le i\le r.
\]
If $N^{1-\theta(n,r)+\varepsilon}\le H\le\frac{N}{\lnc}$, then
\begin{equation*}
J_{n,r}(N,H)=\frac{2^r r^{r-\frac{r}{n}}\gamma(n,r)}{n^r}\mathfrak S(N)\frac{H^{r-1}}{N^{r-\frac{r}{n}}}+O\left(\frac{H^{r-1}}{N^{r-\frac{r}{n}}}\left(\frac{1}{\lnc^{r-1}}+\frac{H^2}{N^2}\right)\right).
\end{equation*}
\end{corollary}

\section{Auxiliary lemmas}
\begin{lemma}\label{lemmaDirichle} {\rm \cite{Karatsuba}.} \textbf{(Dirichlet's rational approximation).}
For every real number \(\alpha\) and every parameter \(P\ge 1\),
there exist coprime integers \(a\) and \(q\), $1\le q\le P,$
such that
\[
\left|\alpha-\frac{a}{q}\right|\le \frac{1}{qP}.
\]
\end{lemma}
\begin{lemma} \label{Lemma formula summirov Puassona}  {\rm \cite{Vaughan-1985}}. 
Let \(f\) be a real-valued function on \([a,b]\) such that
\(f'(u)\) is monotonic. If, for every $u\in[a,b]$ and some integers \(H_1,H_2\), one has $H_1\le f'(u)\le H_2,$
then
\[
\sum_{a<n\le b} e(f(n))
=
\sum_{H_1\le h\le H_2}
\int\limits_a^b e(f(u)-hu)\,du
+
O(\ln(H+2)),
\]
where \(H=\max(|H_1|,|H_2|)\).
\end{lemma}
\vspace{2pt}
 \begin{lemma}\label{Lemma otsenka trig integ po pervoy proizv}{\rm \cite{Arkhipov}}. 
 Let \(f\) be a real-valued function with monotonic derivative
\(f'(u)\), and let \(g\) be a monotonic function on \([a,b]\).
Suppose that
\[
|f'(u)|\ge m_1>0,
\qquad
|g(u)|\le M
\]
for every \(u\in[a,b]\). Then
\[
\int\limits_a^b g(u)e(f(u)) du
\ll
\frac{M}{m_1}.
\]
 \end{lemma}\vspace{2pt}

\begin{lemma}\label{orsumm} {\rm \cite{Arkhipov}.}
Let $n \ge 3$ be an integer, and let
\[
    f(x)=a_nx^n+\cdots+a_1x+a_0
\]
be a polynomial with integer coefficients. Let $q$ also be a natural number and
\[
    (a_n,\ldots,a_1,q)=1.
\]
Then
\[
    |\sum\limits_{x=1}^{q}e\Big(\frac{f(x)}{q}\Big)| \le c(n) q^{1-\frac{1}{n}},
\]
where
\[
    c(n)=
    \begin{cases}
        e^{4n}, & n\ge 10,\\[2mm]
        e^{nA(n)}, & 3\le n\le 9,
    \end{cases}
\]
and
\[
\begin{aligned}
    A(3)&=6.1, & A(4)&=5.5, & A(5)&=5, & A(6)&=4.7,\\
    A(7)&=4.4, & A(8)&=4.2, & A(9)&=4.05.
\end{aligned}
\]
\end{lemma}

\begin{lemma} \label{Lemma otsenka trig int po n-oy proizvodn}{\rm \cite{Arkhipov}}. Let \(f(u)\) be a real function with a derivative of order \(n>1\)
on the interval \([a,b]\). Suppose there exists a constant \(A>0\) such that
$A\le |f^{(n)}(u)|$
for all \(u\in[a,b]\). Then the following estimate holds:
$$
\left|\int\limits_a^be(f(u))du\right|\le \min (b-a,6nA^{-\frac{1}{n}}).
$$
\end{lemma}
\begin{lemma}\label{Lemma-Weyl-3} {\rm \cite{Rakh-2018}}. Let $x\ge x_0>0$, $y_0<y\le 0.01x$, $\tau(\cdot)$~--~the divisor function, $\alpha$~--~a real number,
$$
\left|\alpha-\frac{a}{q}\right|\le \frac{1}{q^2}, \qquad (a,q)=1,
$$
then the following estimate holds:
\begin{align*}
|T(\alpha;x,y)|\le 2y\left(4n!\left(\frac{1}{q}+\frac{1}{y}+\frac{q\ln q}{y^n}\right)\max_{h<y^{n-1}}\tau(h)\right)^{\frac{1}{2^{n-1}}},
\end{align*}
\end{lemma}
\begin{lemma}[{\cite[Theorem~3]{Rakh-2024}}]
\label{lem:short-Weyl-Hua}
Let $x$ and $y$ be positive integers
such that
\[
\sqrt{x}<y\le\frac{x}{\log x}.
\]
Then, for every integer $k$ with
$1\le k\le n$, we have
\[
\int\limits_0^1 |T(\alpha;x,y)|^{2^k}\,d\alpha
\ll y^{2^k-k+\varepsilon}.
\]
\end{lemma}
\begin{lemma}[A modern Weyl estimate]
\label{Lemma-Weyl-VMVT}
{\rm\cite[Theorem~5, equation~(3.2)]{Bourgain-2016}.}
Let $n,N\in\mathbb N$ with $n\ge 3$, and let $a_1,\ldots,a_n$
be real coefficients. Fix an integer $k$ with $2\le k\le n$.
Suppose that
$$
\left|a_k-\frac{a}{q}\right|\le\frac{1}{q^2}
$$
for some coprime integers $a,q$ with $q\ge 1$.
Then for every $\varepsilon>0$ we have
$$
\left|\sum\limits_{m=1}^{N}e(a_1m+\cdots+a_nm^n)\right|
\ll N^{1+\varepsilon}
\left(\frac{1}{q}+\frac{1}{N}+\frac{q}{N^k}\right)^{\frac{1}{n(n-1)}}.
$$
\end{lemma}

\begin{lemma}{\rm ~\cite[p.~174]{UVatson}.}\label{vatson} For $m>0$ and $n\in\mathbb N$, we have the identity
\begin{align*}
&\int\limits_0^{\infty}\frac{\sin^n(mt)}{t^n}dt=\frac{\pi
m^{n-1}}{2^n(n-1)!}\left[n^{n-1}-\frac{n}{1!}(n-2)^{n-1}+\right.\\
 &\left.+\frac{n(n-1)}{2!}(n-4)^{n-1}
-\frac{n(n-1)(n-2)}{3!}(n-6)^{n-1}+\ldots\right].
\end{align*}
\end{lemma}

\begin{lemma}[A large-denominator estimate for short Weyl sums]
\label{Lemma-large-denominators}
Let $n\ge3$ be fixed and let $0<\rho\le\frac{1}{3}$ be fixed.
Suppose that $y$ is sufficiently large and
\[
100y\le x\le y^{2-\rho}.
\]
Put
\[
Q=y^{\rho},\qquad
\tau=\frac{x^{n-2}y^2}{Q},\qquad
\sigma=\frac{\rho}{8n^2(n-1)}.
\]
Suppose that $a\in\mathbb Z$, $q\in\mathbb N$ satisfy
\[
(a,q)=1,\qquad Q<q\le \tau,\qquad
|q\alpha-a|\le \tau^{-1}.
\]
Then
\[
\sum_{x-y<m\le x}e(\alpha m^n)
\ll y^{1-\sigma}.
\]
\end{lemma}

\begin{proof}
We apply Lemma~\ref{Lemma-Weyl-VMVT}. Put
\[
\omega=\frac{1}{n(n-1)},\qquad
t=\frac{\rho}{4n},\qquad
\sigma=\frac{t\omega}{2},\qquad
\varepsilon=\frac{t\omega}{4}=\frac{\sigma}{2}.
\]
Write $u=[x-y]$ and $N=[x]-[x-y]$, so that
$u\asymp x$ and $N\asymp y$.
Shifting the summation variable by $u$, we obtain
\[
\left|\sum_{x-y<m\le x}e(\alpha m^n)\right|=|F|,
\]
where
\[
\begin{aligned}
F&=\sum_{m=1}^{N}e\left(a_1m+a_2m^2+\cdots+a_nm^n\right)=\sum_{m=1}^{N}e(a_1m)\,
  e\left(\sum_{k=2}^{n}a_km^k\right),
\end{aligned}
\qquad
a_k=\binom{n}{k}u^{n-k}\alpha.
\]
For each $k$ with $2\le k\le n$, apply Lemma~\ref{lemmaDirichle}
to $a_k$ with $P_k=y^{k-t}$. Thus there exist coprime integers
$b_k$ and $q_k$ such that
\[
1\le q_k\le P_k,\qquad
\left|a_k-\frac{b_k}{q_k}\right|
\le \frac{1}{q_kP_k}
\le \frac{1}{q_k^2}.
\]
In particular, $|q_ka_k-b_k|\le y^{t-k}$.
If $q_k>y^t$ for some $k$, then, since $N\asymp y$,
\[
\frac{1}{q_k}+\frac{1}{N}+\frac{q_k}{N^k}\ll y^{-t}.
\]
Lemma~\ref{Lemma-Weyl-VMVT} therefore gives
\[
|F|\ll y^{1+\varepsilon-t\omega}
=y^{1-\frac{3\sigma}{2}}\ll y^{1-\sigma},
\]
as required. 

Suppose now that $q_k\le y^t$ for every
$k=2,\ldots,n$. By the definition of $a_k$ and the identity
$$
k\binom{n}{k}=(n-k+1)\binom{n}{k-1},
$$ 
we have
$$
ku a_k=(n-k+1)a_{k-1}\qquad (3\le k\le n).
$$
Hence, for each such $k$,
\begin{align*}
&\bigl|ku b_kq_{k-1}-(n-k+1)b_{k-1}q_k\bigr|\le
ku q_{k-1}|b_k-q_ka_k|
+(n-k+1)q_k|q_{k-1}a_{k-1}-b_{k-1}|\\
&\quad\le kx y^{2t-k}+(n-k+1)y^{2t-k+1}\le (n+1)y^{2t-1}<1.
\end{align*}
The integer inside the absolute value must therefore be zero. Thus
\[
\frac{b_{k-1}}{q_{k-1}}
=\frac{k}{n-k+1}\cdot u\cdot\frac{b_k}{q_k}
\qquad (3\le k\le n).
\]
Applying this relation successively for $k=3,\ldots,n$, we obtain
\[
\frac{b_2}{q_2}
=\left(\prod\limits_{k=3}^{n}\frac{k}{n-k+1}\right)
 u^{n-2}\frac{b_n}{q_n}
=\binom{n}{2}u^{n-2}\frac{b_n}{q_n}.
\]

Subtracting this equality from $a_2=\binom{n}{2}u^{n-2}\alpha$
gives
\[
a_2-\frac{b_2}{q_2}
=\binom{n}{2}u^{n-2}
 \left(\alpha-\frac{b_n}{q_n}\right).
\]
It follows from the bound for $|q_2a_2-b_2|$ and
$u=[x-y]\ge\frac{x}{2}$ that
\[
\left|\alpha-\frac{b_n}{q_n}\right|
=\frac{|q_2a_2-b_2|}{q_2\binom{n}{2}u^{n-2}}
\le \frac{y^{t-2}}{q_2\binom{n}{2}u^{n-2}}
\le \frac{2^{n-2}y^{t-2}}{q_2x^{n-2}}.
\]

We now compare $\frac{b_n}{q_n}$ with the original fraction
$\frac{a}{q}$. 

Since both fractions are in lowest terms and
$q_n\le y^t<y^\rho=Q<q$, they are distinct.
Hence 
$$
|aq_n-b_nq|\ge 1.
$$
Using $|q\alpha-a|\le\frac{1}{\tau}$, we obtain
\begin{align*}
1&\le |aq_n-b_nq|
\le q_n|a-q\alpha|
     +qq_n\left|\alpha-\frac{b_n}{q_n}\right|\le \frac{y^t}{\tau}
     +\tau y^t\frac{2^{n-2}y^{t-2}}{q_2x^{n-2}}\\
&< y^{t-\rho}+\frac{2^{n-2}}{q_2}y^{2t-\rho}\le \frac{1+2^{n-2}}{y^{\rho-2t}}<1.
\end{align*}
The last inequality
holds for sufficiently large $y$, since $t=\frac{\rho}{4n}$
and hence $\rho-2t>0$.

This contradiction shows that $q_k>y^t$ for at least one
$k\in\{2,\ldots,n\}$. The estimate already proved for that case
completes the proof.
\end{proof}

\begin{lemma}[A mean-value estimate for short Weyl sums]
\label{Lemma-moment-VMVT}
Let $n\ge3$ be fixed, let $x$ be sufficiently large, and suppose that
$x^{\frac{1}{2}}\le y\le0.01x$. Then, for every $\eta>0$,
\[
\int\limits_0^1|T(\alpha;x,y)|^{n(n+1)}\,d\alpha
\ll x^{1-n}y^{n(n+1)-1+\eta}.
\]
\end{lemma}

\begin{proof}
Apply \cite[Proposition~1]{Kumchev-Liu-2017} with $k=n$,
$s=n(n+1)$, and 
$$
\lambda(m)=
\begin{cases}
1, & m\le x,\\
0, & m>x.
\end{cases}
$$
Taking $\varepsilon=\frac{\eta}{2}$ in the proposition and using
$x^{\frac{\eta}{2}}\le y^\eta$, which follows from $y\ge x^{\frac{1}{2}}$,
we obtain
\[
\int\limits_0^1 |T(\alpha;x,y)|^{n(n+1)}\,d\alpha
\ll x^{1-n+\varepsilon}y^{n(n+1)-1}\ll x^{1-n}y^{n(n+1)-1+\eta}.
\]
\end{proof}

\vspace{4pt}

\section{\bf Proof of Theorem \ref{ShortExposumIR}}
We first derive estimate~\eqref{ExpsumIR05} without imposing the bounds
on $\lambda$, extending the argument developed in \cite{Mirzo-2026}
for quartic sums to an arbitrary degree $n$.

We have
\[
T\left(\frac{a}{q}+\lambda;x,y\right)
=
\sum_{x-y<k\le x}
e\left(\frac{ak^n}{q}+\lambda k^n\right).
\]
We use the expansion
\[
e\left(\frac{ak^n}{q}\right)
=
\frac{1}{q}
\sum_{b=1}^{q}
S_b(a,q)e\left(-\frac{bk}{q}\right),
\]
where
\[
S_b(a,q)
=
\sum_{v=1}^{q}
e\left(\frac{av^n+bv}{q}\right).
\]
It follows that
\[
T\left(\frac{a}{q}+\lambda;x,y\right)
=
\frac{1}{q}
\sum_{b=1}^{q}
S_b(a,q) T_b,
\]
where
\begin{equation}
\label{ExpsumIR01}
T_b
=
\sum_{x-y<k\le x}
e\left(\lambda k^n-\frac{b}{q}k\right).
\end{equation}

Put
\[
f_b(u)=\lambda u^n-\frac{b}{q}u.
\]
Then
\[
f_b'(u)=n\lambda u^{n-1}-\frac{b}{q}.
\]
Since $f_b'(u)$ is monotonic on the interval $[x-y,x]$, Lemma \ref{Lemma formula summirov Puassona} gives
\begin{equation}
\label{ExpsumIR02}
T_b
=
\sum_{H_{1b}\le h\le H_{2b}}\ 
\int\limits_{x-y}^{x}
e\left(
\lambda u^n-\left(\frac{b}{q}+h\right)u
\right)du
+
O\left(\ln(H_b+2)\right),
\end{equation}
where \(H_{1b}\) is the greatest integer and \(H_{2b}\) is the least
integer such that
\[
H_{1b}\le f_b'(u)\le H_{2b}
\]
for all \(u\in[x-y,x]\), and
\[
H_b=\max\left\{|H_{1b}|,|H_{2b}|\right\}.
\]

For all $u\in[x-y,x]$,
\[
|f_b'(u)|
\le
n|\lambda|u^{n-1}+\frac{b}{q}\le n|\lambda|x^{n-1}+\frac{b}{q}.
\]
Since \(b\le q\), the definitions of \(H_{1b}\) and
\(H_{2b}\) imply
\[
H_b\ll 1+|\lambda|x^{n-1}.
\]

In (\ref{ExpsumIR02}), make the substitution
\[
m=b+qh.
\]
Then
\[
\frac{b}{q}+h=\frac{m}{q},
\]
and the corresponding integral takes the form
\[
I_m
=
\int\limits_{x-y}^{x}
e\left(\lambda u^n-\frac{m}{q}u\right)\,du.
\]
Since $m\equiv b\pmod q$, we have
\[
S_b(a,q)=S_m(a,q),
\]
where
\[
S_m(a,q)
=
\sum_{k=1}^{q}
e\left(\frac{ak^n+mk}{q}\right).
\]
Since \((a,q)=1\), applying Lemma~\ref{orsumm} to
\(S_m(a,q)\), we obtain
\begin{equation}
\label{ExpsumIR03}
S_m(a,q)\ll q^{1-\frac{1}{n}}.
\end{equation}

Let
\[
\mathcal M
=
\left\{
m=b+qh:
1\le b\le q,\quad
H_{1b}\le h\le H_{2b}
\right\}.
\]
For $1\le b\le q$, the representation of every $m\in\mathcal M$
in the form $m=b+qh$ is unique. Therefore, by (\ref{ExpsumIR01})--(\ref{ExpsumIR03}),
\begin{equation}
\label{ExpsumIR04}
T\left(\frac{a}{q}+\lambda;x,y\right)
\ll
q^{-\frac{1}{n}}
\sum_{m\in\mathcal M}|I_m|
+
q^{1-\frac{1}{n}}
\ln\left(2+|\lambda|x^{n-1}\right).
\end{equation}

We now estimate
\[
\sum_{m\in\mathcal M}|I_m|.
\]
Put
\[
F_m(u)=\lambda u^n-\frac{m}{q}u.
\]
Then
\[
F_m'(u)=n\lambda u^{n-1}-\frac{m}{q},
\qquad
F_m''(u)=n(n-1)\lambda u^{n-2}.
\]
Since $y\le 0.01x$, we have $u\asymp x$ on $[x-y,x]$, and hence
\begin{equation*}
|F_m''(u)|
\asymp
|\lambda|x^{n-2}.
\end{equation*}

Consequently, if a stationary point lies in the interval
\[
x-y\le u\le x,
\]
then the corresponding value \(m\) must belong to the interval
\[
\Delta=
\left[
\min\left(nq\lambda(x-y)^{n-1},\;nq\lambda x^{n-1}\right),
\max\left(nq\lambda(x-y)^{n-1},\;nq\lambda x^{n-1}\right)
\right].
\]
Conversely, if \(m\in\Delta\), then the equation
\[
m=nq\lambda u^{n-1}
\]
has a solution \(u\in[x-y,x]\), that is, the integral \(I_m\) has a stationary point.
Consequently, the integral \(I_m\) has a stationary point on the interval
\([x-y,x]\) if and only if \(m\in\Delta\).
The length of \(\Delta\) satisfies the estimate
\[
|\Delta|
\ll
q|\lambda|\bigl(x^{n-1}-(x-y)^{n-1}\bigr).
\]
Since
\[
x^{n-1}-(x-y)^{n-1}\ll x^{n-2}y,
\]
we have
\[
|\Delta|\ll q|\lambda|x^{n-2}y.
\]

We distinguish the set of stationary and nearly stationary values
\[
\mathcal M_0=
\{m\in\mathcal M:\operatorname{dist}(m,\Delta)\le 1\},
\]
where \(\operatorname{dist}(m,\Delta)\) denotes the distance from the point \(m\)
to the interval \(\Delta\).
If \(\Delta=[\Delta_1,\Delta_2]\), then
\[
\operatorname{dist}(m,\Delta)=
\begin{cases}
0, & \Delta_1\le m\le \Delta_2,\\
\Delta_1-m, & m<\Delta_1,\\
m-\Delta_2, & m>\Delta_2.
\end{cases}
\]

Then
\[
|\mathcal M_0|\ll 1+|\Delta|
\ll 1+q|\lambda|x^{n-2}y.
\]
For \(m\in\mathcal M_0\), we apply the second-derivative estimate in Lemma~\ref{Lemma otsenka trig int po n-oy proizvodn}. Since
\[
|F_m''(u)|\asymp |\lambda|x^{n-2},
\]
we obtain
\[
|I_m|\ll (|\lambda|x^{n-2})^{-\frac{1}{2}}.
\]
Consequently,
\[
\sum_{m\in\mathcal M_0}|I_m|
\ll
(1+q|\lambda|x^{n-2}y)(|\lambda|x^{n-2})^{-\frac{1}{2}}.
\]

Now consider \(m\in\mathcal M\setminus\mathcal M_0\).
Put
\[
\rho=[\operatorname{dist}(m,\Delta)],
\]
where \([t]\) denotes the integer part of \(t\).

Then \(\rho\ge 1\). For all \(u\in[x-y,x]\), the quantity
\(nq\lambda u^{n-1}\) belongs to the interval \(\Delta\); therefore,
\[
|nq\lambda u^{n-1}-m|\ge \rho.
\]
Consequently,
\[
|F_m'(u)|
=
\left|n\lambda u^{n-1}-\frac{m}{q}\right|
=
\frac{1}{q} |nq\lambda u^{n-1}-m|
\ge
\frac{\rho}{q}.
\]
Since \(F_m'(u)\) is monotonic on the interval \([x-y,x]\), Lemma \ref{Lemma otsenka trig integ po pervoy proizv} gives
\[
|I_m|\ll \frac{q}{\rho}.
\]

We next show that \(\rho\le q\). By the definitions of \(H_{1b}\) and
\(H_{2b}\) and the continuity of \(f_b'\), for every
\(m=b+qh\in\mathcal M\), there exist \(u_m\in[x-y,x]\) and a real
number \(\delta_m\), with \(|\delta_m|\le1\), such that
\[
h=f_b'(u_m)+\delta_m.
\]

Consequently,
\[
\begin{aligned}
m
&=b+qh=b+q\bigl(f_b'(u_m)+\delta_m\bigr)=b+q\bigl(n\lambda u_m^{n-1}-\frac{b}{q}\bigr)+q\delta_m=nq\lambda u_m^{n-1}+q\delta_m.
\end{aligned}
\]
Since $nq\lambda u_m^{n-1}\in\Delta$ and $|q\delta_m|\le q,$
it follows that 
\[
\Delta_1-q\le m\le\Delta_2+q.
\]

Consequently, the remaining values
\(m\in\mathcal M\setminus\mathcal M_0\) lie in the two intervals
\[
\Delta_1-q\le m<\Delta_1-1
\]
or
\[
\Delta_2+1<m\le\Delta_2+q.
\]

Hence, \(1\le \rho\le q\). For each fixed \(\rho\), there are at
most two values \(m\in\mathcal M\setminus\mathcal M_0\) corresponding
to it, one on each side of \(\Delta\). Therefore,
\[
\sum_{m\in\mathcal M\setminus\mathcal M_0}|I_m|
\ll
\sum_{1\le \rho\le q}\frac{q}{\rho}
\ll q\ln(2q).
\]

Thus,
\[
\sum_{m\in\mathcal M}|I_m|
\ll
q\ln(2q)
+
(1+q|\lambda|x^{n-2}y)(|\lambda|x^{n-2})^{-\frac{1}{2}}.
\]
Since
\[
\ln(2q)
\le
\ln q+\ln\left(2+|\lambda|x^{n-1}\right),
\]
substituting the preceding estimate into (\ref{ExpsumIR04}), we obtain
the following estimate for \(T\):
\begin{equation}
\label{ExpsumIR05}
T\left(\frac{a}{q}+\lambda;x,y\right)
\ll
q^{1-\frac{1}{n}}\ln q
+
q^{-\frac{1}{n}}(|\lambda|x^{n-2})^{-\frac{1}{2}}
+
q^{1-\frac{1}{n}}y(|\lambda|x^{n-2})^{\frac{1}{2}}
+
q^{1-\frac{1}{n}}\ln(2+|\lambda|x^{n-1}).
\end{equation}

No restriction on $\lambda$ other than $\lambda\neq0$ has been used
in deriving \eqref{ExpsumIR05}. We now impose the bounds on $\lambda$
in Theorem~\ref{ShortExposumIR}.

From the lower bound \(
|\lambda|\gg \frac{1}{qx^{n-2}y}
\)
it follows that
\(
|\lambda|x^{n-2}\gg \frac{1}{qy}.
\)
Therefore,
\begin{equation}
\label{ExpsumIR06}
q^{-\frac{1}{n}}(|\lambda|x^{n-2})^{-\frac{1}{2}}
\ll
q^{-\frac{1}{n}}(qy)^{\frac{1}{2}}
=
q^{\frac{1}{2}-\frac{1}{n}}y^{\frac{1}{2}}.
\end{equation}

On the other hand, from the upper bound
\(
|\lambda|\ll \frac{1}{qy^{n-1}}
\)
we obtain
\(
|\lambda|x^{n-2}\ll \frac{x^{n-2}}{qy^{n-1}}.
\)
Consequently,
\begin{equation}
\label{ExpsumIR07}
q^{1-\frac{1}{n}}y(|\lambda|x^{n-2})^{\frac{1}{2}}
\ll
q^{1-\frac{1}{n}}y
\left(\frac{x^{n-2}}{qy^{n-1}}\right)^{\frac{1}{2}}
=
q^{\frac{1}{2}-\frac{1}{n}}\cdot\frac{x^{\frac{n-2}{2}}}{y^{\frac{n-3}{2}}}.
\end{equation}

Since $y\le x$, we have
\[
y^{\frac{1}{2}}
\le
\frac{x^{\frac{n-2}{2}}}{y^{\frac{n-3}{2}}}.
\]

Hence the right-hand side of (\ref{ExpsumIR06}) is absorbed by the
right-hand side of (\ref{ExpsumIR07}). It now follows from
(\ref{ExpsumIR05})--(\ref{ExpsumIR07}) that
\[
T\left(\frac{a}{q}+\lambda;x,y\right)
\ll
q^{1-\frac{1}{n}}\ln q
+
q^{1-\frac{1}{n}}
\ln\left(2+|\lambda|x^{n-1}\right)
+
q^{\frac{1}{2}-\frac{1}{n}}
x^{\frac{n-2}{2}}y^{\frac{3-n}{2}}.
\]
This completes the proof of Theorem~\ref{ShortExposumIR}.

\section{\bf Proof of Theorem \ref{TeorAsForWaringPPSl}}
Let $n\ge3$, $r=r(n)$, and
$N^{1-\theta(n,r)+\varepsilon}\le H\le\frac{N}{\lnc}$, where
$0<\varepsilon<\theta(n,r)$ is fixed. We use one circle-method
decomposition for all $n$, choosing its parameters according to the degree.

Relabelling the variables and their corresponding coefficients, we may
assume that $\mu_1\le\cdots\le\mu_r$. For $1\le k\le r$, put
\[
N_k=(\mu_kN+H)^{\frac{1}{n}},\qquad
H_k=(\mu_kN+H)^{\frac{1}{n}}-(\mu_kN-H)^{\frac{1}{n}}.
\]

We distinguish the two cases in the definition of $r(n)$
and choose $\delta$ accordingly.
\[
r=
\begin{cases}
2^n+1, & n=3,4,\\
n(n+1)+1, & n\ge5,
\end{cases}
\qquad
\delta=
\begin{cases}
\displaystyle \frac{2}{r-1}, & n=3,4,\\[6pt]
\displaystyle \min\left\{\frac{1}{4},n\varepsilon\right\}, & n\ge5.
\end{cases}
\]
For $n=3,4$, this gives
$\theta(n,r)=\frac{\delta}{n(n-1+\delta)}$.
For $n\ge5$, set $\sigma=\frac{\delta}{8n^2(n-1)}$.

Choose the parameters
\[
Q=\begin{cases}
\frac{H_r}{\lnc}, & n=3,4,\\[3pt]
H_r^\delta, & n\ge5,
\end{cases}
\quad
\tau=\begin{cases}
2n(n-1)N_1^{n-2}H_1, & n=3,4,\\[3pt]
\frac{N_r^{n-2}H_r^2}{Q}, & n\ge5,
\end{cases}
\quad
\eta_q=\begin{cases}
\frac{1}{2nqN_r^{n-1}}, & n=3,4,\\[3pt]
\frac{1}{q\tau Q^{\frac{1}{n}}}, & n\ge5.
\end{cases}
\]

By orthogonality, the number of solutions of
\eqref{Formula-x1n+...+xrn=N} under the conditions
\eqref{formula |xin-N/n|<=H} is
\begin{align*}
J_{n,r}(N,H)&=\int\limits_0^1e(-\alpha N)\prod\limits_{k=1}^r\sum\limits_{\substack{m\in\mathbb N\\|m^n-\mu_kN|\le H}}e(\alpha m^n)d\alpha= \int\limits_0^1e(-\alpha N)\prod\limits_{k=1}^r\left(T(\alpha;N_k,H_k)+\theta_k(\alpha)\right)d\alpha,
\end{align*}
where
\[
\theta_k(\alpha)=
\begin{cases}
e\bigl(\alpha(N_k-H_k)^n\bigr),
& \text{if } N_k-H_k\in\mathbb Z,\\
0,
& \text{otherwise}.
\end{cases}
\]
In particular,
\[
|\theta_k(\alpha)|\le1.
\]
The upper endpoint \(N_k\) and the length \(H_k\) of the interval of
summation in \(T(\alpha;N_k,H_k)\) are expressed in terms of \(N\) and
\(H\) by the following asymptotic formulas:
\begin{align*}
N_k&=\mu_{k}^{\frac{1}{n}}N^{\frac{1}{n}}\left(1+\frac{H}{\mu_kN}\right)^{\frac{1}{n}}=\mu_{k}^{\frac{1}{n}}N^{\frac{1}{n}}\left(1+O\left(\frac{H}{N}\right)\right),\\
H_k&=\mu_{k}^{\frac{1}{n}}N^{\frac{1}{n}}\left(\left(1+\frac{H}{\mu_kN}\right)^{\frac{1}{n}}-\left(1-\frac{H}{\mu_kN}\right)^{\frac{1}{n}}\right)=
\frac{2H}{n\mu_k^{1-\frac{1}{n}}N^{1-\frac{1}{n}}}\left(1+O\left(\frac{H^2}{N^2}\right)\right). 
\end{align*}
Hence
\[
N_k\asymp N^{\frac{1}{n}},\qquad
H_k\asymp\frac{H}{N^{1-\frac{1}{n}}}
\qquad (1\le k\le r).
\]
For $n\ge5$, we have $\theta(n,r)=\frac{1}{2n}$, so the lower
and upper bounds on $H$ give
\[
H_k\gg N^{\frac{1}{2n}+\varepsilon}
\asymp N_k^{\frac{1}{2}+n\varepsilon},
\qquad
\frac{H_k}{N_k}\asymp\frac{H}{N}\le\frac{1}{\lnc}.
\]
Since $\delta\le n\varepsilon$ and $\delta\le\frac{1}{4}$,
\[
\left(\frac{1}{2}+n\varepsilon\right)(2-\delta)
\ge\left(\frac{1}{2}+\delta\right)(2-\delta)
=1+\delta\left(\frac{3}{2}-\delta\right)>1.
\]
Hence, for sufficiently large $N$,
\begin{equation}\label{eq:short-interval-range}
100H_k\le N_k\le H_k^{2-\delta}
\qquad (1\le k\le r).
\end{equation}

For $\nu=1,2,\ldots,r$ and $1\le i_1<\ldots<i_\nu\le r$, we introduce the notation
$$
\mathscr{D}_\nu=\mathscr{D}(i_1,\ldots,i_\nu)=\{1,2,\ldots,r\}\setminus\{i_1,\ldots,i_\nu\}
$$
and use the identity
\begin{align*}
\prod\limits_{k=1}^r\left(T(\alpha;N_k,H_k)+\theta_k(\alpha)\right)=&\prod\limits_{k=1}^rT(\alpha;N_k,H_k)\\
&+\sum_{\nu=1}^{r-1}\sum_{1\le i_1<\ldots<i_\nu\le r}\prod\limits_{j=1}^\nu T(\alpha;N_{i_j},H_{i_j})
\prod\limits_{k\in\mathscr{D}_\nu} \theta_k(\alpha)+\prod\limits_{k=1}^r\theta_k(\alpha).
\end{align*}
The contribution of the last term in this identity is zero. Indeed, it
vanishes unless \(N_k-H_k\in\mathbb Z\) for every \(k\). If this condition
holds, then
\[
\sum_{k=1}^r(N_k-H_k)^n
=
\sum_{k=1}^r(\mu_kN-H)
=
N-rH.
\]
Thus, \(rH\) is a positive integer, and hence
\[
\int\limits_0^1
e(-\alpha N)\prod\limits_{k=1}^r\theta_k(\alpha)\,d\alpha
=
\int\limits_0^1
e(-\alpha rH)\,d\alpha
=
0.
\]
Therefore, we represent \(J_{n,r}(N,H)\) as
\begin{align}
&J_{n,r}(N,H)=\int\limits_0^1e(-\alpha N)\prod\limits_{k=1}^rT(\alpha;N_k,H_k)d\alpha+R_{n,r}(N,H),\label{formula J_{n,r}(N,H)=int...+R1(n,H)}\\
&R_{n,r}(N,H)
=\sum_{\nu=1}^{r-1}
\sum_{1\le i_1<\cdots<i_\nu\le r}
\int\limits_0^1
e(-\alpha N)
\prod\limits_{j=1}^{\nu}T(\alpha;N_{i_j},H_{i_j})
\prod\limits_{k\in\mathscr D_\nu}\theta_k(\alpha)\,d\alpha. \nonumber
\end{align}
Each term in $R_{n,r}(N,H)$ contains $\nu\le r-1$ factors $T$,
and $|\theta_k(\alpha)|\le1$. Applying the AM-GM inequality to $z_1^{r-1},\ldots,z_\nu^{r-1}$ and $r-1-\nu$
additional factors equal to $1$, we obtain
\[
\prod_{j=1}^{\nu}z_j
\le 1+\sum_{j=1}^{\nu}z_j^{r-1}
\qquad (z_j\ge0,\ 1\le\nu\le r-1).
\]
Taking $z_j=|T(\alpha;N_{i_j},H_{i_j})|$, we extend the sum on the
right from the selected indices $i_1,\ldots,i_\nu$ to all indices
$1\le i\le r$, since its terms are nonnegative. Integrating over
$0\le\alpha\le1$ gives the same upper bound for every choice of
indices. Summing over all $\sum\limits_{\nu=1}^{r-1}\binom{r}{\nu}=2^r-2$
such choices and absorbing this fixed factor into the implied constant,
we obtain
\begin{equation}\label{eq:remainder-moments}
|R_{n,r}(N,H)|
\ll 1+\sum_{k=1}^{r}
\int\limits_0^1 |T(\alpha;N_k,H_k)|^{r-1}\,d\alpha.
\end{equation}
For $n=3,4$, we have $r-1=2^n$. Since
$1-n\theta(n,r)>\frac{1}{2}$ and $H\le\frac{N}{\lnc}$,
we may partition the integer points of $(N_k-H_k,N_k]$ into a bounded
number of intervals with integer endpoints satisfying the hypotheses
of Lemma~\ref{lem:short-Weyl-Hua}. Applying the lemma on each interval,
we obtain
\begin{equation}\label{eq:low-degree-moment}
\int\limits_0^1 |T(\alpha;N_k,H_k)|^{r-1}\,d\alpha
\ll H_k^{r-1-n+\varepsilon}
\qquad (1\le k\le r).
\end{equation}

Substituting \eqref{eq:low-degree-moment} into
\eqref{eq:remainder-moments} and using
$H_k\asymp \frac{H}{N^{1-\frac{1}{n}}}$, we obtain
\begin{align*}
&R_{n,r}(N,H)\ll \left(\frac{H}{N^{1-\frac{1}{n}}}\right)^{r-1-n+\varepsilon}=\frac{H^{r-1}}{N^{r-\frac{r}{n}}}\cdot \frac{\Big(N^{1-\frac{1}{n}}\Big)^{n+1-\varepsilon}}{H^{n-\varepsilon}}\\
&\ll
\frac{H^{r-1}}{N^{r-\frac{r}{n}}\lnc^{r-1}}\cdot \frac{N^{\frac{n^2-1}{n}-\frac{n-1}{n}\varepsilon}}{\Big(N^{1-\theta(n,r)+\varepsilon}\Big)^{n-\varepsilon}}\lnc^{r-1}=\frac{H^{r-1}}{N^{r-\frac{r}{n}}\lnc^{r-1}}\cdot N^{-(\frac{1}{n}-n\theta(n,r))-(n-\frac{1}{n}+\theta(n,r))\varepsilon+\varepsilon^2}\lnc^{r-1}.
\end{align*}
Since
$$
\frac{1}{n}-n\theta(n,r)
=\frac{(r-3)(n-1)}{n\Big((r-1)(n-1)+2\Big)}>0.
$$
for every sufficiently small fixed \(\varepsilon>0\), we have
$$
-\Big(\frac{1}{n}-n\theta(n,r)\Big)-\Big(n-\frac{1}{n}+\theta(n,r)\Big)\varepsilon+\varepsilon^2<0.
$$
Consequently,
\begin{align*}
R_{n,r}(N,H)& \ll\frac{H^{r-1}}{N^{r-\frac{r}{n}}\lnc^{r-1}}.
\end{align*}
For $n\ge 5$, we have $r-1=n(n+1)$. By
\eqref{eq:short-interval-range},
\[
N_k^{\frac{1}{2}}\le H_k\le 0.01N_k
\qquad (1\le k\le r).
\]
Hence Lemma~\ref{Lemma-moment-VMVT} gives, for every $\eta>0$,
\begin{equation}\label{eq:moment-Tk}
\int\limits_0^1 |T(\alpha;N_k,H_k)|^{r-1}\,d\alpha
\ll N_k^{1-n}H_k^{r-2+\eta}
\qquad (1\le k\le r).
\end{equation}
Substituting \eqref{eq:moment-Tk} with $0<\eta<1$ into
\eqref{eq:remainder-moments} and using
$N_k\asymp N^{\frac{1}{n}}$ and
$H_k\asymp \frac{H}{N^{1-\frac{1}{n}}}$, we obtain
\[
|R_{n,r}(N,H)|
\ll 1+N^{\frac{1-n}{n}}
\left(\frac{H}{N^{1-\frac{1}{n}}}\right)^{r-2+\eta}.
\]
Since $\theta(n,r)=\frac{1}{2n}$, the lower bound on $H$ gives
\[
\frac{H}{N^{1-\frac{1}{n}}}\ge N^{\frac{1}{2n}+\varepsilon}>N^{\frac{1}{2n}}.
\]
As $r>2n$, the second term in the bound for $|R_{n,r}(N,H)|$
is at least $N^{\frac{r-2n+\eta}{2n}}>1$. Hence
\begin{align*}
|R_{n,r}(N,H)|
&\ll N^{\frac{1-n}{n}}
\left(\frac{H}{N^{1-\frac{1}{n}}}\right)^{r-2+\eta}=\frac{H^{r-1}}{N^{r-\frac{r}{n}}}
\left(\frac{N^{1-\frac{1}{n}}}{H}\right)^{1-\eta}\\
&\le \frac{H^{r-1}}{N^{r-\frac{r}{n}}\lnc^{r-1}}
\cdot N^{-\left(\frac{1}{2n}+\varepsilon\right)(1-\eta)}
\lnc^{r-1}
\ll \frac{H^{r-1}}{N^{r-\frac{r}{n}}\lnc^{r-1}}.
\end{align*}

Thus, for every $n\ge3$, \eqref{formula J_{n,r}(N,H)=int...+R1(n,H)}
yields
\begin{equation}\label{formula J_{n,r}(N,H)=int...+O(H(r-r/n+varisilon))}
J_{n,r}(N,H)
=\int\limits_0^1 e(-\alpha N)
  \prod_{k=1}^r T(\alpha;N_k,H_k)\,d\alpha
+O\left(\frac{H^{r-1}}{N^{r-\frac{r}{n}}\lnc^{r-1}}\right).
\end{equation}
Applying Lemma~\ref{lemmaDirichle} with $P=\tau$, for every
\[
\alpha\in\left[-\frac{1}{\tau},1-\frac{1}{\tau}\right)
\]
there exist coprime integers $a$ and $q$ such that
\[
\alpha=\frac{a}{q}+\lambda,\qquad
1\le q\le\tau,\qquad
|\lambda|\le\frac{1}{q\tau}.
\]
We may choose $a$ so that $0\le a\le q-1$. Since $(a,q)=1$,
the case $a=0$ occurs only when $q=1$.

For $1\le q\le Q$ and $0\le a\le q-1$ with $(a,q)=1$, put
\[
\mathfrak{M}(a,q)=
\left\{
\alpha\in\left[-\frac{1}{\tau},1-\frac{1}{\tau}\right):
\left|\alpha-\frac{a}{q}\right|\le\frac{1}{q\tau}
\right\}.
\]
Define
\[
\mathfrak{M}=
\bigcup_{1\le q\le Q}\ 
\bigcup_{\substack{0\le a\le q-1\\(a,q)=1}}
\mathfrak{M}(a,q),
\qquad
\mathfrak{m}=
\left[-\frac{1}{\tau},1-\frac{1}{\tau}\right)
\setminus\mathfrak{M}.
\]
The intervals $\mathfrak{M}(a,q)$ are pairwise disjoint.
We further decompose $\mathfrak{M}$ into the sets
$\mathfrak{M}_1$ and $\mathfrak{M}_2$, where
\[
\mathfrak{M}_1=
\bigcup_{1\le q\le Q}
\bigcup_{\substack{0\le a\le q-1\\(a,q)=1}}
\mathfrak{M}_1(a,q),
\qquad
\mathfrak{M}_1(a,q)=
\left\{
\alpha\in\left[-\frac{1}{\tau},1-\frac{1}{\tau}\right):
\left|\alpha-\frac{a}{q}\right|\le\eta_q
\right\},
\]
and $\mathfrak{M}_2=\mathfrak{M}\setminus\mathfrak{M}_1$.

Denoting by $J(\mathfrak{M}_1)$, $J(\mathfrak{M}_2)$, and $J(\mathfrak{m})$ the integrals over the sets $\mathfrak{M}_1$, $\mathfrak{M}_2$, and $\mathfrak{m}$, respectively, taking into account (\ref{formula J_{n,r}(N,H)=int...+O(H(r-r/n+varisilon))}), we obtain
\begin{equation}\label{formula J_{n,r}(N,H)=I(M1)+I(M2)+I(m)+O(..)}
J_{n,r}(N,H)=J(\mathfrak{M}_1)+J(\mathfrak{M}_2)+J(\mathfrak{m})+O\left(\frac{H^{r-1}}{N^{r-\frac{r}{n}}\lnc^{r-1}}\right).
\end{equation}
The integral over \(\mathfrak M_1\) supplies the main term,
whereas the integrals over \(\mathfrak M_2\) and \(\mathfrak m\)
are absorbed into the error term.

\subsection{The integral $J(\mathfrak M_1)$}
By the definition of \(\mathfrak{M}_1\), writing
\(\alpha=\frac{a}{q}+\lambda\) on each arc, we obtain
\[
J(\mathfrak{M}_1)
=\sum_{q\le Q}
\sum_{\substack{0\le a<q\\(a,q)=1}}\ 
\int\limits_{|\lambda|\le\eta_q}
\left(
\prod_{k=1}^{r}
T\!\left(\frac{a}{q}+\lambda;N_k,H_k\right)
\right)
e\!\left(-\left(\frac{a}{q}+\lambda\right)N\right)
d\lambda.
\]

In \cite{Rakh-2024}, an asymptotic formula for the integral over
$\mathfrak M_1$ was obtained for $r=2^n+1$. We use the same
approximation on the central arcs for both choices of $r$.

For $n=3,4$, the definition of $\eta_q$ gives
$q\eta_q N_r^{n-1}=\frac{1}{2n}$. For $n\ge5$, we have
$N_r\le H_r^{2-\delta}=\frac{H_r^2}{Q}$ by (\ref{eq:short-interval-range}), and hence
\[
q\eta_q N_r^{n-1}
=\frac{N_r^{n-1}}{\tau Q^{\frac{1}{n}}}
=\frac{N_r Q^{1-\frac{1}{n}}}{H_r^2}
\le Q^{-\frac{1}{n}}
\le\frac{1}{2n}
\]
for sufficiently large $N$. Since $N_k\le N_r$, it follows in
both cases that
\[
|\lambda|\le\eta_q\le\frac{1}{2nqN_k^{n-1}}
\qquad (1\le k\le r).
\]
We apply the corollary with approximation parameter $2nN_k^{n-1}$.
By the preceding bound on $|\lambda|$ and the inequality
$(n-1)H_k\le N_k$, its hypotheses are satisfied for sufficiently
large $N$.

Thus, by \cite[Corollary~2]{Rakh-2024},
\[
T\left(\frac{a}{q}+\lambda;N_k,H_k\right)
=\frac{S(a,q)}{q}v_k(\lambda)
+O\left(q^{\frac{1}{2}+\varepsilon}\right),
\qquad
v_k(\lambda)=\int\limits_{N_k-H_k}^{N_k}e(\lambda u^n)\,du.
\]
For $\lambda\ne 0$, apply
Lemma~\ref{Lemma otsenka trig integ po pervoy proizv}
with $f(u)=\lambda u^n$ and $g(u)=1$. On $[N_k-H_k,N_k]$,
the derivative $f'(u)=n\lambda u^{n-1}$ is monotone and satisfies
$|f'(u)|\gg |\lambda|N_k^{n-1}$. Hence
\[
|v_k(\lambda)|\ll\frac{1}{|\lambda|N_k^{n-1}}.
\]
Combining this with the trivial bound $|v_k(\lambda)|\le H_k$ (with equality at \(\lambda=0\)) 
and using $N_k^{n-1}H_k\asymp H$ and $H_k\asymp H_r$, we obtain
\begin{align*}
|v_k(\lambda)|
&\ll \min\!\left\{H_k,\frac{1}{|\lambda|N_k^{n-1}}\right\}
 \asymp H_k\min\!\left\{1,\frac{1}{H|\lambda|}\right\} = \frac{H_k}{\max\{1,H|\lambda|\}}
 \le \frac{2H_k}{1+H|\lambda|}
 \ll \frac{H_r}{1+H|\lambda|}.
\end{align*}
Put
\[
v=\frac{H_r}{q^{\frac{1}{n}}(1+H|\lambda|)},
\qquad
w=q^{\frac{1}{2}+\varepsilon}.
\]
Since \(\bigl|\frac{S(a,q)}{q}v_k(\lambda)\bigr|\ll v\)
for each \(k\), we obtain
$$
\prod_{k=1}^{r}
T\left(\frac{a}{q}+\lambda;N_k,H_k\right)
=
\prod_{k=1}^{r}\left(\frac{S(a,q)}{q}v_k(\lambda)+O(w)\right)=
\left(\frac{S(a,q)}{q}\right)^{r}
\prod_{k=1}^{r}v_k(\lambda)
+
O\left(
\sum_{j=1}^{r}v^{r-j}w^j
\right).
$$
Using the inequality
\[
\sum_{j=1}^{r}v^{r-j}w^j
= w\sum_{j=1}^{r}v^{r-j}w^{j-1}
\le r\,w\max\{v,w\}^{r-1}
\ll v^{r-1}w+w^r
\qquad (v,w\ge 0),
\]
we see that the error term in the product is \(O(v^{r-1}w+w^r)\).

For each \(q\), there are at most \(q\) admissible values of \(a\).
Moreover, since \(r>2\),
\[
\int\limits_{-\eta_q}^{\eta_q}
\frac{d\lambda}{(1+H|\lambda|)^{r-1}}
\le \frac{2}{H}\int\limits_0^\infty\frac{dt}{(1+t)^{r-1}}
= \frac{2}{H(r-2)}.
\]
We estimate the contributions of \(v^{r-1}w\) and \(w^r\) separately.
For the first, the preceding integral estimate gives
\[
\sum_{q\le Q} q\int\limits_{-\eta_q}^{\eta_q} v^{r-1}w\,d\lambda
\ll \frac{H_r^{r-1}}{H}
\sum_{q\le Q} q^{\frac{3}{2}+\varepsilon-\frac{r-1}{n}}
\ll \frac{H_r^{r-1}}{H},
\]
since \(\frac{r-1}{n}\ge\frac{8}{3}\) and 
 \(\frac{3}{2}+\varepsilon-\frac{r-1}{n}<-1\) for
\(\varepsilon<\frac{1}{6}\).

To estimate the contribution of \(w^r\), we first bound \(q\eta_q\).
For \(n=3,4\), the definition of \(\eta_q\) and the relation
\(N_r^{n-1}H_r\asymp H\) give
\[
q\eta_q=\frac{1}{2nN_r^{n-1}}\asymp\frac{H_r}{H}.
\]
For \(n\ge 5\), we have \(N_r\le H_r^{2-\delta}\) and
\(Q=H_r^\delta\), so \(N_rQ\le H_r^2\). Hence
\[
q\eta_q
=\frac{1}{\tau Q^{\frac{1}{n}}}
=\frac{1}{N_r^{n-1}Q^{\frac{1}{n}}}
 \cdot\frac{N_rQ}{H_r^2}
\le \frac{1}{N_r^{n-1}Q^{\frac{1}{n}}}
\ll \frac{H_r}{H Q^{\frac{1}{n}}}.
\]
Thus \(q\eta_q\ll\frac{H_r}{H}\) in both cases.

Since \(w=q^{\frac{1}{2}+\varepsilon}\) is independent of \(\lambda\),
we integrate over \([-\eta_q,\eta_q]\) and use the preceding bound
together with \(Q\le H_r\) to obtain
$$
\sum_{q\le Q} q\int\limits_{-\eta_q}^{\eta_q} w^r\,d\lambda= 2\sum_{q\le Q}(q\eta_q)\,q^{\frac{r}{2}+r\varepsilon} \ll \frac{H_r}{H}\,Q^{1+\frac{r}{2}+r\varepsilon}
\le \frac{H_r^{2+\frac{r}{2}+r\varepsilon}}{H}
\le \frac{H_r^{r-1}}{H},
$$
where the last inequality follows from
\(2+\frac{r}{2}+r\varepsilon<2+\frac{2r}{3}\le r-1\) for \(r\ge 9\).
Thus the total error is \(O\!\left(\frac{H_r^{r-1}}{H}\right)\).

Substituting the main term of the product into the integral over
\(\mathfrak{M}_1\) and using the preceding error estimate, we obtain
\[
\begin{aligned}
J(\mathfrak{M}_1)
&=
\sum_{q\le Q}
\sum_{\substack{0\le a<q\\(a,q)=1}}
e\!\left(-\frac{aN}{q}\right)
\left(\frac{S(a,q)}{q}\right)^r
\int\limits_{-\eta_q}^{\eta_q}
e(-N\lambda)\prod_{k=1}^{r}v_k(\lambda)\,d\lambda
+O\!\left(\frac{H_r^{r-1}}{H}\right).
\end{aligned}
\]

We first extend each integral to the real line and then extend the sum
over \(q\) to infinity. Since
\[
\int\limits_{|\lambda|>\eta_q}
\frac{d\lambda}{(1+H|\lambda|)^r}
=
\frac{2}{H(r-1)}(1+H\eta_q)^{1-r},
\]
the bounds for \(v_k(\lambda)\) and \(S(a,q)\) show that the error
from these two extensions is
\[
\ll \frac{H_r^r}{H}
\left(
\sum_{q\le Q}q^{1-\frac{r}{n}}(1+H\eta_q)^{1-r}
+
\sum_{q>Q}q^{1-\frac{r}{n}}
\right).
\]
Since \(\frac{r}{n}\ge 3\) and \(r-1\ge 2\), we have
\[
q^{1-\frac{r}{n}}(1+H\eta_q)^{1-r}
\le \frac{1}{q^2(1+H\eta_q)^2}
= \frac{1}{(q+Hq\eta_q)^2}.
\]
For \(n=3,4\), we have
\[
Hq\eta_q=\frac{H}{2nN_r^{n-1}}\asymp H_r\gg Q,
\]
while for \(n\ge 5\),
\[
Hq\eta_q
=\frac{HQ^{1-\frac{1}{n}}}{N_r^{n-2}H_r^2}
\asymp \frac{N_r}{H_r}Q^{1-\frac{1}{n}}
\gg Q^{1-\frac{1}{n}}.
\]
Thus \(Hq\eta_q\gg Q^{1-\frac{1}{n}}\) in both cases, and hence
$$
\sum_{q\le Q}q^{1-\frac{r}{n}}(1+H\eta_q)^{1-r}
\ll \sum_{q=1}^{\infty}
     \frac{1}{\bigl(q+Q^{1-\frac{1}{n}}\bigr)^2} \le \int\limits_0^{\infty}
     \frac{dt}{\bigl(t+Q^{1-\frac{1}{n}}\bigr)^2}
= \frac{1}{Q^{1-\frac{1}{n}}}.
$$
For the second sum, since \(\frac{r}{n}\ge 3\), we have
\[
\sum_{q>Q}q^{1-\frac{r}{n}}
\ll \int\limits_Q^\infty t^{1-\frac{r}{n}}\,dt
= \frac{Q^{2-\frac{r}{n}}}{\frac{r}{n}-2}
\ll \frac{1}{Q}.
\]
Thus the error from extending the integral and the sum is
\[
\ll \frac{H_r^r}{H}
\left(\frac{1}{Q^{1-\frac{1}{n}}}+\frac{1}{Q}\right)
\ll \frac{H_r^r}{H Q^{1-\frac{1}{n}}}.
\]
Combining this with the error \(O\!\left(\frac{H_r^{r-1}}{H}\right)\)
from the expansion of the product, we obtain a total error of
\[
\ll \frac{H_r^r}{H}
\left(\frac{1}{H_r}+\frac{1}{Q^{1-\frac{1}{n}}}\right)
\ll \frac{H_r^r}{H\lnc^{r-1}},
\]
since \(H_r\ge Q\) and \(Q\gg N^c\) for some fixed \(c>0\).

The sum over all \(q\) is the singular series \(\mathfrak{S}(N)\).
Using \(\frac{H_r^r}{H}\asymp\frac{H^{r-1}}{N^{r-\frac{r}{n}}}\),
we conclude that
\begin{equation}\label{Jm1}
J(\mathfrak{M}_1)
=
\mathfrak{S}(N)
\int\limits_{-\infty}^{\infty}
e(-N\lambda)\prod_{k=1}^{r}v_k(\lambda)\,d\lambda
+
O\!\left(
\frac{H^{r-1}}{N^{r-\frac{r}{n}}\lnc^{r-1}}
\right).
\end{equation}
To evaluate the remaining integral, we put \(u^n=\mu_kN+Ht\)
in \(v_k(\lambda)\). Since
\[
(N_k-H_k)^n=\mu_kN-H,
\qquad
N_k^n=\mu_kN+H,
\]
the new variable \(t\) ranges over \([-1,1]\). Writing
\(\beta=2\pi H\lambda\), we obtain
\[
v_k(\lambda)
=
\frac{2H\,e(\mu_kN\lambda)}
     {n(\mu_kN)^{1-\frac{1}{n}}}
\cdot
\frac{1}{2}\int\limits_{-1}^{1}
e^{i\beta t}
\left(1+\frac{Ht}{\mu_kN}\right)^{\frac{1}{n}-1}\,dt.
\]
Since \(H\le\frac{N}{\lnc}\) and the \(\mu_k\) are fixed,
Taylor's formula gives, uniformly for \(|t|\le 1\),
\[
\left(1+\frac{Ht}{\mu_kN}\right)^{\frac{1}{n}-1}
=
1+\left(\frac{1}{n}-1\right)\frac{Ht}{\mu_kN}
+O\!\left(\frac{H^2}{N^2}\right).
\]
Substituting this expansion into the preceding integral, we obtain
\[
\begin{aligned}
v_k(\lambda)
=&
\frac{2H\,e(\mu_kN\lambda)}{n(\mu_kN)^{1-\frac{1}{n}}}
\left(
\frac{1}{2}\int\limits_{-1}^{1}e^{i\beta t}\,dt
+
\left(\frac{1}{n}-1\right)\frac{H}{\mu_kN}
\frac{1}{2}\int\limits_{-1}^{1}t e^{i\beta t}\,dt
+
O\!\left(\frac{H^2}{N^2}\right)
\right).
\end{aligned}
\]
The error is uniform in \(\beta\), since \(|e^{i\beta t}|=1\).
Writing
\[
K(\beta)=\frac{1}{2}\int\limits_{-1}^{1}e^{i\beta t}\,dt
=\frac{\sin\beta}{\beta},
\qquad K(0)=1,
\]
and differentiating under the integral sign, we find that
\[
\frac{1}{2}\int\limits_{-1}^{1}t e^{i\beta t}\,dt
=\frac{K'(\beta)}{i}.
\]
Hence
\[
v_k(\lambda)
=
\frac{2H\,e(\mu_kN\lambda)}{n(\mu_kN)^{1-\frac{1}{n}}}
\left(
K(\beta)
+
\left(\frac{1}{n}-1\right)\frac{H}{i\mu_kN}K'(\beta)
+
O\!\left(\frac{H^2}{N^2}\right)
\right).
\]
Substituting the above expression for \(v_k(\lambda)\) into
\eqref{Jm1}, we use the identity
\[
e(-N\lambda)\prod\limits_{k=1}^{r}e(\mu_kN\lambda)
=
e\!\left(N\lambda\left(\sum\limits_{k=1}^{r}\mu_k-1\right)\right)
=1.
\]
Taking the constant factors outside the integral and putting
\(\beta=2\pi H\lambda\), we obtain
\[
\begin{aligned}
J(\mathfrak M_1)
=&
\frac{2^r\mathfrak S(N)}{n^r}
\frac{H^{r-1}}{N^{r-\frac{r}{n}}}
\left(\prod\limits_{k=1}^{r}\mu_k^{\frac{1}{n}-1}\right)
\\
&\times\frac{1}{2\pi}
\int\limits_{-\infty}^{\infty}
\prod\limits_{k=1}^{r}
\left(
K(\beta)
+\left(\frac{1}{n}-1\right)\frac{H}{i\mu_kN}K'(\beta)
+O\!\left(\frac{H^2}{N^2}\right)
\right)\,d\beta+
O\!\left(\frac{H^{r-1}}{N^{r-\frac{r}{n}}\lnc^{r-1}}\right).
\end{aligned}
\]
To expand the product under the integral, we use the identity
\[
\prod\limits_{k=1}^{r}(b+a_k)
=b^r+\sum\limits_{\nu=1}^{r}
b^{r-\nu}
\sum\limits_{1\le i_1<\cdots<i_\nu\le r}
\prod\limits_{j=1}^{\nu}a_{i_j}.
\]
In our case \(b=K(\beta)\), and the formula for \(v_k(\lambda)\)
gives
\[
\begin{aligned}
a_k=
\left(\frac{1}{n}-1\right)
\frac{H}{i\mu_kN}K'(\beta)
+O\!\left(\frac{H^2}{N^2}\right).
\end{aligned}
\]
Hence the product expands as follows:
\[
\begin{aligned}
\prod\limits_{k=1}^{r}\bigl(K(\beta)+a_k\bigr)
&=K^r(\beta)
+K^{r-1}(\beta)(a_1+\cdots+a_r)
+K^{r-2}(\beta)
\sum\limits_{1\le i<j\le r}a_i a_j
+\cdots+a_1a_2\cdots a_r.
\end{aligned}
\]

For the term containing \(a_1+\cdots+a_r\), we obtain
\[
\begin{aligned}
K^{r-1}(\beta)\sum\limits_{k=1}^{r}a_k
=&
\left(\frac{1}{n}-1\right)\frac{H}{iN}
\left(\sum\limits_{k=1}^{r}\frac{1}{\mu_k}\right)
K^{r-1}(\beta)K'(\beta)+
O\!\left(
\frac{H^2}{N^2(1+|\beta|)^{r-1}}
\right),
\end{aligned}
\]
since \(r\) is fixed and
\(\lvert K(\beta)\rvert\ll\frac{1}{1+|\beta|}\).

Every subsequent term in the expansion contains at least one
pair \(a_i a_j\). To estimate such a pair, we use two bounds
for \(a_j\). Its expansion above gives
\[
|a_j|\ll\frac{H}{N},
\]
whereas its exact expression in terms of \(v_j(\lambda)\) gives
\[
|a_j|
\le
\frac{n(\mu_jN)^{1-\frac{1}{n}}}{2H}|v_j(\lambda)|
+|K(\beta)|
\ll \frac{1}{1+|\beta|}.
\]
Expanding \(a_i\) alone and using
\(\lvert K'(\beta)\rvert\ll\frac{1}{1+|\beta|}\), we obtain
\[
\begin{aligned}
|a_i a_j|
&\ll
\frac{H}{N(1+|\beta|)}\,|a_j|
+\frac{H^2}{N^2}\,|a_j|\ll
\frac{H}{N(1+|\beta|)}\frac{H}{N}
+\frac{H^2}{N^2}\frac{1}{1+|\beta|}
\ll
\frac{H^2}{N^2(1+|\beta|)}.
\end{aligned}
\]
Each of the remaining \(r-2\) factors is either \(K(\beta)\)
or some \(a_k\), and hence is
\(O\!\left(\frac{1}{1+|\beta|}\right)\).
Thus the double sum and all subsequent terms together contribute
\[
O\!\left(\frac{H^2}{N^2(1+|\beta|)^{r-1}}\right).
\]
Consequently,
\[
\begin{aligned}
\prod\limits_{k=1}^{r}\bigl(K(\beta)+a_k\bigr)
=&
K^r(\beta)
+
\left(\frac{1}{n}-1\right)\frac{H}{iN}
\left(\sum\limits_{k=1}^{r}\frac{1}{\mu_k}\right)
K^{r-1}(\beta)K'(\beta)+
O\!\left(\frac{H^2}{N^2(1+|\beta|)^{r-1}}\right).
\end{aligned}
\]
We now integrate this equality with respect to \(\beta\).
The term of first order vanishes, because \(K(\beta)\) is even
and \(K'(\beta)\) is odd. Hence
\[
\int\limits_{-\infty}^{\infty}
K^{r-1}(\beta)K'(\beta)\,d\beta=0.
\]
The integrated error is \(O\!\left(\frac{H^2}{N^2}\right)\),
since \(r>2\) and
\[
\int\limits_{-\infty}^{\infty}
\frac{d\beta}{(1+|\beta|)^{r-1}}
=\frac{2}{r-2}.
\]
Since $K(\beta)$ is even, the integral in the main term is
\[
\frac{1}{2\pi}\int\limits_{-\infty}^{\infty}K^r(\beta)\,d\beta
=
\frac{1}{\pi}\int\limits_{0}^{\infty}
\frac{\sin^r\!\beta}{\beta^r}\,d\beta.
\]
Applying Lemma~\ref{vatson} with $m=1$ and $n=r$, and recalling
that $r$ is odd, we obtain
\[
\frac{1}{\pi}\int\limits_{0}^{\infty}
\frac{\sin^r\!\beta}{\beta^r}\,d\beta
=
\frac{1}{2^r(r-1)!}
\sum\limits_{j=0}^{\frac{r-1}{2}}
(-1)^j\binom{r}{j}(r-2j)^{r-1}
=
\gamma(n,r).
\]

Substituting this value into the expression for $J(\mathfrak M_1)$
and using $\mathfrak S(N)\ll 1$, we obtain
\begin{equation}\label{formula I(M1)-3}
J(\mathfrak M_1)
=\frac{2^r\gamma(n,r)}{n^r}
\left(\prod\limits_{k=1}^{r}\mu_k^{\frac{1}{n}-1}\right)
\mathfrak S(N)\frac{H^{r-1}}{N^{r-\frac{r}{n}}}+
O\!\left(
\frac{H^{r-1}}{N^{r-\frac{r}{n}}}
\left(\frac{1}{\lnc^{r-1}}+\frac{H^2}{N^2}\right)
\right).
\end{equation}

\subsection{The integral $J(\mathfrak M_2)$}

For $n=3,4$, we use the estimates obtained in \cite{Rakh-2024}.
In that paper, two cases are considered according to the size of
$|\lambda|$. We show that these estimates remain valid for our choice of $H$.

Fix $1\le\nu\le r$. Isolating the $\nu$th factor and applying
the arithmetic--geometric mean inequality to the remaining $r-1$
factors, we obtain, by periodicity,
\[
|J(\mathfrak M_2)|
\le
\max_{\alpha\in\mathfrak M_2}|T(\alpha;N_\nu,H_\nu)|
\frac{1}{r-1}
\sum_{\substack{1\le k\le r\\ k\ne\nu}}
\int\limits_0^1 |T(\alpha;N_k,H_k)|^{r-1}\,d\alpha.
\]
Since $r-1=2^n$ and $H_k\asymp\frac{H}{N^{1-\frac{1}{n}}}$,
the moment bound \eqref{eq:low-degree-moment} gives
\begin{equation*}
J(\mathfrak M_2)\ll
\left(\frac{H}{N^{1-\frac{1}{n}}}\right)^{r-1-n+\varepsilon}\!
\max_{\alpha\in\mathfrak M_2}|T(\alpha;N_\nu,H_\nu)|=
\frac{H^{r-1}}{N^{r-\frac{r}{n}}}\cdot
\frac{N^{n-\frac{1}{n}-\frac{n-1}{n}\varepsilon}}
     {H^{n-\varepsilon}}\!
\max_{\alpha\in\mathfrak M_2}|T(\alpha;N_\nu,H_\nu)|.
\end{equation*}

\medskip
\noindent

\textbf{Case 1.} Suppose that
\[
\eta_q<|\lambda|
\le
\frac{1}{2nqN_\nu^{\,n-1}}.
\]
In this case, \cite{Rakh-2024} gives the estimate
\[
J(\mathfrak M_2)
\ll
\frac{H^{r-1}}
     {N^{\,r-\frac{r}{n}}\lnc^{\,1-\frac{1}{n}}}
\left(
\frac{N^{1-\delta_1+\eta_1(n)}}{H}
\right)^{d_1},
\]
where
\[
\delta_1=\frac{1}{n^3-n^2+n},
\qquad
d_1=n-1+\frac{1}{n}-\varepsilon,
\]
and
\[
\eta_1(n)
=
\frac{(n^2-n)\varepsilon}
     {(n^2-n+1)(n^2-n+1-n\varepsilon)}
\le
\frac{\varepsilon}{n^2-n+1}
<
\frac{\varepsilon}{2}.
\]

For $n=3,4$, we have $\delta\le\frac{1}{n}$, and hence
$\theta(n,r)\le\delta_1$. Therefore,
\[
1-\delta_1+\eta_1(n)
\le
1-\theta(n,r)+\frac{\varepsilon}{2}.
\]
Taking into account the condition
\[
H\ge N^{1-\theta(n,r)+\varepsilon},
\]
we obtain
\[
\frac{N^{1-\delta_1+\eta_1(n)}}{H}
\le
N^{-\frac{\varepsilon}{2}}.
\]
Consequently,
\begin{equation*}
  J(\mathfrak M_2)
\ll
\frac{H^{r-1}}
     {N^{\,r-\frac{r}{n}}\lnc^{\,1-\frac{1}{n}}}
N^{-\frac{\varepsilon d_1}{2}}
\ll
\frac{H^{r-1}}
     {N^{\,r-\frac{r}{n}}\lnc^{\,r-1}}.
\end{equation*}

\medskip
\noindent
\textbf{Case 2.} Suppose that
\[
\frac{1}{2nqN_\nu^{\,n-1}}
<
|\lambda|
\le
\frac{1}{q\tau}.
\]
In this case, \cite{Rakh-2024} gives the estimate
\[
J(\mathfrak M_2)
\ll
\frac{H^{r-1}}
     {N^{\,r-\frac{r}{n}}
      \lnc^{\,\frac{1}{2}-\frac{1}{n}}}
\left(
\frac{N^{1-\delta_2+\eta_2(n)}}{H}
\right)^{d_2},
\]
where
\[
\delta_2=
\frac{1}{n^3-\frac{1}{2}n^2+n},
\qquad
d_2=n-\frac{1}{2}+\frac{1}{n}-\varepsilon,
\]
and
\[
\eta_2(n)
=
\frac{\left(n^2-\frac{1}{2}n\right)\varepsilon}
     {\left(n^2-\frac{1}{2}n+1\right)
      \left(n^2-\frac{1}{2}n+1-n\varepsilon\right)}
\le
\frac{\varepsilon}{n^2-n+1}
<
\frac{\varepsilon}{2}.
\]

For $n=3,4$, the inequality $\theta(n,r)\le\delta_2$ follows from
$\delta\le\frac{2(n-1)}{n(2n-1)}$, or equivalently
$2n^2-n\le2^n(n-1)$. Hence,
\[
1-\delta_2+\eta_2(n)
\le
1-\theta(n,r)+\frac{\varepsilon}{2},
\]
and therefore
\[
\frac{N^{1-\delta_2+\eta_2(n)}}{H}
\le
N^{-\frac{\varepsilon}{2}}.
\]
It follows that
\[
J(\mathfrak M_2)
\ll
\frac{H^{r-1}}
     {N^{\,r-\frac{r}{n}}
      \lnc^{\,\frac{1}{2}-\frac{1}{n}}}
N^{-\frac{\varepsilon d_2}{2}}
\ll
\frac{H^{r-1}}
     {N^{\,r-\frac{r}{n}}\lnc^{\,r-1}}.
\]
For $n\ge5$, every $\alpha\in\mathfrak M_2$ can be written as
\[
\alpha=\frac{a}{q}+\lambda,\qquad
(a,q)=1,\qquad 1\le q\le Q,\qquad
\eta_q<|\lambda|\le\frac{1}{q\tau}.
\]
Applying \eqref{ExpsumIR05} with $x=N_r$ and $y=H_r$ gives
\begin{align*}
|T(\alpha;N_r,H_r)|
&\ll q^{1-\frac{1}{n}}\ln q
   + q^{-\frac{1}{n}}
     \bigl(|\lambda|N_r^{n-2}\bigr)^{-\frac{1}{2}} 
   + q^{1-\frac{1}{n}}H_r
     \bigl(|\lambda|N_r^{n-2}\bigr)^{\frac{1}{2}}
   + q^{1-\frac{1}{n}}\ln\bigl(2+|\lambda|N_r^{n-1}\bigr).
\end{align*}

Since $1\le q\le Q$, the first term satisfies
\[
q^{1-\frac{1}{n}}\ln q
\le Q^{1-\frac{1}{n}}\ln Q.
\]

For the second and third terms, the inequalities
$\eta_q<|\lambda|\le\frac{1}{q\tau}$, together with
\[
\eta_q=\frac{1}{q\tau Q^{\frac{1}{n}}},
\qquad
\tau=\frac{N_r^{n-2}H_r^2}{Q},
\]
give
\[
\frac{Q^{1-\frac{1}{n}}}{qH_r^2}
< |\lambda|N_r^{n-2}
\le \frac{Q}{qH_r^2}.
\]
Using the lower bound for the second term, we obtain
\[
q^{-\frac{1}{n}}
\bigl(|\lambda|N_r^{n-2}\bigr)^{-\frac{1}{2}}
\le H_r Q^{-\frac{1}{2n}}
    \left(\frac{q}{Q}\right)^{\frac{1}{2}-\frac{1}{n}}
\le H_r Q^{-\frac{1}{2n}}.
\]
Using the upper bound for the third term, we obtain
\[
q^{1-\frac{1}{n}}H_r
\bigl(|\lambda|N_r^{n-2}\bigr)^{\frac{1}{2}}
\le q^{\frac{1}{2}-\frac{1}{n}}Q^{\frac{1}{2}}
\le Q^{1-\frac{1}{n}}.
\]

For the fourth term, \eqref{eq:short-interval-range} gives
$QN_r\le H_r^2$. Hence
\[
|\lambda|N_r^{n-1}
\le \frac{QN_r}{qH_r^2}
\le \frac{1}{q}\le 1,
\]
and therefore
\[
q^{1-\frac{1}{n}}\ln\bigl(2+|\lambda|N_r^{n-1}\bigr)
\le Q^{1-\frac{1}{n}}\ln 3.
\]

Combining the four estimates, we obtain
\[
|T(\alpha;N_r,H_r)|
\ll H_r Q^{-\frac{1}{2n}}
   + Q^{1-\frac{1}{n}}\ln Q.
\]
Since $Q=H_r^\delta$ and $\frac{\delta}{2n}>\sigma$, the first
term is at most $H_r^{1-\sigma}$. For the second term, the strict
inequality $\delta\left(1-\frac{1}{n}\right)<1-\sigma$ gives
\[
H_r^{\delta\left(1-\frac{1}{n}\right)}\ln Q
\ll H_r^{1-\sigma}.
\]
Thus
\[
|T(\alpha;N_r,H_r)|\ll H_r^{1-\sigma}.
\]

Using the bound for $|T(\alpha;N_r,H_r)|$ on $\mathfrak M_2$
and then applying H\"older's inequality with all $r-1$ exponents
equal to $r-1$, we obtain
\begin{align*}
|J(\mathfrak M_2)|
&\le \int\limits_{\mathfrak M_2}
      \prod\limits_{k=1}^{r}|T(\alpha;N_k,H_k)|\,d\alpha\ll H_r^{1-\sigma}
     \int\limits_{\mathfrak M_2}
     \prod\limits_{k=1}^{r-1}|T(\alpha;N_k,H_k)|\,d\alpha\\
&\le H_r^{1-\sigma}
      \prod\limits_{k=1}^{r-1}
      \left(
        \int\limits_{\mathfrak M_2}
        |T(\alpha;N_k,H_k)|^{r-1}\,d\alpha
      \right)^{\frac{1}{r-1}}\le H_r^{1-\sigma}
      \prod\limits_{k=1}^{r-1}
      \left(
        \int\limits_0^1
        |T(\alpha;N_k,H_k)|^{r-1}\,d\alpha
      \right)^{\frac{1}{r-1}}.
\end{align*}
By \eqref{eq:moment-Tk} with $\eta=\frac{\sigma}{2}$,
and since $N_k\asymp N_r$ and $H_k\asymp H_r$, it follows that
\begin{align*}
|J(\mathfrak M_2)|
&\ll H_r^{1-\sigma}N_r^{1-n}H_r^{r-2+\frac{\sigma}{2}}= N_r^{1-n}H_r^{r-1}H_r^{-\frac{\sigma}{2}}
\ll \frac{H^{r-1}}{N^{r-\frac{r}{n}}}
    H_r^{-\frac{\sigma}{2}}.
\end{align*}
Since $H_r\asymp\frac{H}{N^{1-\frac{1}{n}}}
\gg N^{\frac{1}{2n}+\varepsilon}$, we have
\[
H_r^{-\frac{\sigma}{2}}\ll\frac{1}{\lnc^{r-1}}.
\]
Together with the estimates for $n=3,4$, this gives, for every $n\ge3$,
\begin{equation}\label{formula I(M2)-4}
J(\mathfrak M_2)
\ll
\frac{H^{r-1}}{N^{r-\frac{r}{n}}\lnc^{r-1}}.
\end{equation}
\subsection{Estimation of the integral $J(\mathfrak{m})$}
First suppose that $n=3,4$. Arguing as for $J(\mathfrak M_2)$,
we obtain, for every fixed $1\le\nu\le r$,
\begin{equation}
\label{J-m-general}
J(\mathfrak{m})
\ll
\frac{H^{r-1}}{N^{r-\frac{r}{n}}}
\cdot\frac{
N^{n-\frac{1}{n}-\frac{n-1}{n}\varepsilon}
}{
H^{n-\varepsilon}
}
\max_{\alpha\in\mathfrak{m}}
\left|T(\alpha;N_\nu,H_\nu)\right|,
\end{equation}
where $\varepsilon>0$ is an arbitrarily small auxiliary constant.

For every $\alpha\in\mathfrak m$, we apply Lemma \ref{lemmaDirichle} once more, now with the parameter
\[
P=H_\nu^{n-1}.
\]
Then there exist coprime integers $a_1,q_1$ such that
\[
\alpha=\frac{a_1}{q_1}+\beta,
\qquad
1\le q_1\le H_\nu^{n-1},
\qquad
|\beta|\le\frac{1}{q_1H_\nu^{n-1}}.
\]
We divide $\mathfrak m$ into two parts:
\[
\mathfrak m=\mathfrak m_1\cup\mathfrak m_2,
\]
where
\[
\mathfrak m_1
=
\left\{
\alpha\in\mathfrak m:
q_1>Q
\right\},
\qquad
\mathfrak m_2
=
\left\{
\alpha\in\mathfrak m:
q_1\le Q
\right\}.
\]
Accordingly,
\begin{equation}
\label{J-m-split}
J(\mathfrak m)=J(\mathfrak m_1)+J(\mathfrak m_2).
\end{equation}

\medskip
\noindent
\textbf{Estimation of $J(\mathfrak m_1)$.}
If $\alpha\in\mathfrak m_1$, then
\[
q_1>Q,
\qquad
q_1\le H_\nu^{n-1}.
\]
Moreover,
\[
|\beta|
\le
\frac{1}{q_1H_\nu^{n-1}}
\le
\frac{1}{q_1^2}.
\]

Since $H_r\asymp H_\nu$, using the Weyl-type estimate stated in Lemma \ref{Lemma-Weyl-3}, we obtain
\begin{align}
\left|T(\alpha;N_\nu,H_\nu)\right|
&\ll
H_\nu^{1+\varepsilon}
\left(
\frac{1}{H_\nu}
+\frac{1}{q_1}
+\frac{q_1\ln q_1}{H_\nu^n}
\right)^\delta
\ll
H_\nu^{1-\delta+\varepsilon}
\lnc^\delta.
\label{T-m1}
\end{align}
Using
\[
H_\nu\asymp\frac{H}{N^{1-\frac{1}{n}}},
\]
and substituting \eqref{T-m1} into \eqref{J-m-general}, we obtain
\begin{equation}
\label{J-m1-factor}
J(\mathfrak m_1)
\ll
\frac{H^{r-1}}{N^{r-\frac{r}{n}}}
\lnc^\delta
\frac{N^{E_0}}{H^{D_0}}\cdot \frac{H^{2\varepsilon}}{N^{\frac{2(n-1)}{n}\varepsilon}}
\ll
\frac{H^{r-1}}{N^{r-\frac{r}{n}}}
\lnc^\delta
\left(
\frac{N^{\frac{E_0}{D_0}}}{H}
\right)^{D_0}\cdot \frac{H^{2\varepsilon}}{N^{\frac{2(n-1)}{n}\varepsilon}},
\end{equation}
where
\[
D_0=n-1+\delta
\]
and
\[
E_0
=
n-1+\frac{n-1}{n}\delta.
\]
Since $\theta(n,r)=\frac{\delta}{n(n-1+\delta)}$ for $n=3,4$, we have
\[
\frac{E_0}{D_0}
=\frac{n-1+\frac{n-1}{n}\delta}{n-1+\delta}=
1-\theta(n,r).
\]

Then, using $H\ge N^{1-\theta(n,r)+\varepsilon}$ and $H<N$,
we obtain from \eqref{J-m1-factor}
\begin{align}
\label{J-m1-final}
J(\mathfrak m_1)
\ll\frac{H^{r-1}}{N^{r-\frac{r}{n}}}
\lnc^\delta
\left(
\frac{N^{1-\theta(n,r)}}{N^{1-\theta(n,r)+\varepsilon}}
\right)^{n-1+\delta}\cdot \frac{N^{2\varepsilon}}{N^{\frac{2(n-1)}{n}\varepsilon}}&\ll\frac{H^{r-1}}{N^{r-\frac{r}{n}}}
\lnc^\delta N^{-\varepsilon(n-1-\frac{2}{n}+\delta)}\nonumber\\
&\ll
\frac{H^{r-1}}
{N^{r-\frac{r}{n}}\lnc^{r-1}}.
\end{align}

\medskip
\noindent
\textbf{Estimation of $J(\mathfrak m_2)$.}
Let $\alpha\in\mathfrak m_2$. Then
$q_1\le Q$.
Since $\alpha\notin\mathfrak M$, the definition of the major arcs implies
\[
|\beta|>\frac{1}{q_1\tau}.
\]
Moreover, $\tau\asymp N_\nu^{n-2}H_\nu.$
Consequently,
\[
\frac{1}{q_1N_\nu^{n-2}H_\nu}
\ll
|\beta|
\le
\frac{1}{q_1H_\nu^{n-1}}.
\]
Thus, Theorem~\ref{ShortExposumIR} is applicable with
\[
x=N_\nu,
\qquad
y=H_\nu,
\qquad
q=q_1,
\qquad
\lambda=\beta.
\]
It follows that
\begin{align}
\left|T(\alpha;N_\nu,H_\nu)\right|
\ll{}&
q_1^{1-\frac{1}{n}}\ln q_1+
q_1^{1-\frac{1}{n}}
\ln\left(
2+|\beta|N_\nu^{n-1}
\right)+
q_1^{\frac{1}{2}-\frac{1}{n}}
N_\nu^{\frac{n-2}{2}}
H_\nu^{\frac{3-n}{2}}.
\label{T-m2-theorem}
\end{align}
Both logarithms in \eqref{T-m2-theorem} are $O(\lnc)$.
Furthermore,
\[
q_1\le Q
\ll\frac{H_\nu}{\lnc}.
\]
Therefore,
\begin{align}
\left|T(\alpha;N_\nu,H_\nu)\right|
& \ll
H_\nu^{1-\frac{1}{n}}
\lnc^{\frac{1}{n}}+
N_\nu^{\frac{n-2}{2}}
H_\nu^{2-\frac{n}{2}-\frac{1}{n}}
\lnc^{-\frac{1}{2}+\frac{1}{n}}\nonumber\\
&\ll \left(\frac{H}{N^{1-\frac{1}{n}}}\right)^{1-\frac{1}{n}}
\lnc^{\frac{1}{n}}+
N^{\frac{n-2}{2n}}
\left(\frac{H}{N^{1-\frac{1}{n}}}\right)^{2-\frac{n}{2}-\frac{1}{n}}
\lnc^{-\frac{1}{2}+\frac{1}{n}}.
\label{T-m2-final}
\end{align}

For the first term on the right-hand side of \eqref{T-m2-final},
put
\[
D_1=n-1+\frac{1}{n},
\qquad
v_1(n)
=
\frac{1}{n(n^2-n+1)}.
\]

For the second term, put
\[
D_2
=
\frac{3n^2-4n+2}{2n},
\qquad
v_2(n)
=
\frac{2}{n(3n^2-4n+2)}.
\]

Substituting \eqref{T-m2-final} into \eqref{J-m-general}, and using again $N^{1-\theta(n,r)+\varepsilon}\le H< N$, we obtain
\begin{align}
J(\mathfrak m_2)
&\ll
\frac{H^{r-1}}{N^{r-\frac{r}{n}}}
\Bigg\{
\lnc^{\frac{1}{n}}
\left(
\frac{
N^{1-v_1(n)}
}{H}
\right)^{D_1}
+
\lnc^{-\frac{1}{2}+\frac{1}{n}}
\left(
\frac{
N^{1-v_2(n)}
}{H}
\right)^{D_2}
\Bigg\}\cdot \frac{H^{\varepsilon}}{N^{\frac{n-1}{n}\varepsilon}}\nonumber\\
&\ll\frac{H^{r-1}}{N^{r-\frac{r}{n}}}
\Bigg\{
\lnc^{\frac{1}{n}}
\left(
\frac{
N^{1-v_1(n)}
}{N^{1-\theta(n,r)+\varepsilon}}
\right)^{D_1}
+
\lnc^{-\frac{1}{2}+\frac{1}{n}}
\left(
\frac{
N^{1-v_2(n)}
}{N^{1-\theta(n,r)+\varepsilon}}
\right)^{D_2}
\Bigg\}\cdot \frac{N^{\varepsilon}}{N^{\frac{n-1}{n}\varepsilon}}\nonumber\\
&\ll  \frac{H^{r-1}}{N^{r-\frac{r}{n}}}
\Bigg\{
\lnc^{\frac{1}{n}}
\left(
N^{\theta(n,r)-v_1(n)}
N^{-\varepsilon}
\right)^{n-1+\frac{1}{n}}
+
\lnc^{-\frac{1}{2}+\frac{1}{n}}
\left(
N^{\theta(n,r)-v_2(n)}
N^{-\varepsilon}
\right)^{\frac{3n-4}{2}+\frac{1}{n}}
\Bigg\}\cdot N^{\frac{1}{n}\varepsilon}\nonumber\\
&\ll \frac{H^{r-1}}{N^{r-\frac{r}{n}}}
\Bigg\{
\left(
N^{\theta(n,r)-v_1(n)}
N^{-\varepsilon}
\right)^{n-1}
+
\left(
N^{\theta(n,r)-v_2(n)}
N^{-\varepsilon}
\right)^{\frac{3n-4}{2}}
\Bigg\}\lnc^{\frac{1}{n}}
\label{J-m2-factor}
\end{align}

For $n=3,4$, the inequalities $\theta(n,r)\le v_1(n)$ and
$\theta(n,r)\le v_2(n)$ are equivalent, respectively, to
\[
\delta\le\frac{1}{n},\qquad \delta\le\frac{2(n-1)}{n(3n-4)}.
\]
Both hold, since $n\le2^{n-1}$ and $n(3n-4)\le2^n(n-1)$ for these degrees.
Consequently, it follows from \eqref{J-m2-factor} that
\begin{equation}
\label{J-m2-final}
J(\mathfrak m_2)\ll \frac{H^{r-1}}{N^{r-\frac{r}{n}}}\Big(N^{-\varepsilon(n-1)}+N^{-\frac{3n-4}{2}\varepsilon}\Big)\lnc^{\frac{1}{n}}
\ll
\frac{H^{r-1}}
{N^{r-\frac{r}{n}}\lnc^{r-1}}.
\end{equation}

Combining \eqref{J-m-split}, \eqref{J-m1-final}, and
\eqref{J-m2-final} gives the required estimate for $n=3,4$.

Suppose now that $n\ge5$ and $\alpha\in\mathfrak m$. We use the
original approximation $\alpha=\frac{a}{q}+\lambda$ obtained from
Lemma~\ref{lemmaDirichle} with $P=\tau$. Since
$|\lambda|\le\frac{1}{q\tau}$, the inequality $q\le Q$ would imply
$\alpha\in\mathfrak M(a,q)$. Hence
\[
Q<q\le\tau,\qquad |q\alpha-a|\le\frac{1}{\tau}.
\]
By \eqref{eq:short-interval-range}, we have
$100H_r\le N_r\le H_r^{2-\delta}$. Thus
Lemma~\ref{Lemma-large-denominators} applies with
$x=N_r$, $y=H_r$, and $\rho=\delta$, giving
\[
|T(\alpha;N_r,H_r)|\ll H_r^{1-\sigma}.
\]

Since $|e(-N\alpha)|=1$ and
$|T(\alpha;N_r,H_r)|\ll H_r^{1-\sigma}$ uniformly for
$\alpha\in\mathfrak m$, we obtain
\[
|J(\mathfrak m)|
\le \int\limits_{\mathfrak m}
     \prod_{k=1}^{r}|T(\alpha;N_k,H_k)|\,d\alpha
\ll H_r^{1-\sigma}
    \int\limits_{\mathfrak m}
    \prod_{k=1}^{r-1}|T(\alpha;N_k,H_k)|\,d\alpha.
\]

As in the estimate for $J(\mathfrak M_2)$, we bound the last factor
by $H_r^{1-\sigma}$ and apply H\"older's inequality to the remaining
$r-1$ factors. Using \eqref{eq:moment-Tk} with $\eta=\frac{\sigma}{2}$,
we obtain
\begin{equation*}
|J(\mathfrak m)|
\ll H_r^{1-\sigma}
     \prod_{k=1}^{r-1}
     \left(\int\limits_0^1 |T(\alpha;N_k,H_k)|^{r-1}\,d\alpha
     \right)^{\frac{1}{r-1}}\ll H_r^{1-\sigma}N_r^{1-n}H_r^{r-2+\frac{\sigma}{2}}.
\end{equation*}
Finally, since $N_r\asymp N^{\frac{1}{n}}$ and
$H_r\asymp\frac{H}{N^{1-\frac{1}{n}}}$, we have
\[
N_r^{1-n}H_r^{r-1}
\asymp
N^{\frac{1-n}{n}}
\left(\frac{H}{N^{1-\frac{1}{n}}}\right)^{r-1}
=
\frac{H^{r-1}}{N^{r-\frac{r}{n}}}.
\]
Consequently,
\[
|J(\mathfrak m)|
\ll \frac{H^{r-1}}{N^{r-\frac{r}{n}}}H_r^{-\frac{\sigma}{2}}.
\]
Since $H_r\gg N^{\frac{1}{2n}+\varepsilon}$, we have
$H_r^{-\frac{\sigma}{2}}\ll\lnc^{-(r-1)}$.
Together with the estimates for $n=3,4$, this gives, for every $n\ge3$,
\begin{equation}\label{J-m-final}
|J(\mathfrak m)|
\ll \frac{H^{r-1}}{N^{r-\frac{r}{n}}\lnc^{r-1}}.
\end{equation}

Substituting \eqref{formula I(M1)-3}, \eqref{formula I(M2)-4},
and \eqref{J-m-final} into
\eqref{formula J_{n,r}(N,H)=I(M1)+I(M2)+I(m)+O(..)}, we obtain
\eqref{asymptotic-almost-proportional} for every $n\ge3$, with
$r=r(n)$ and $\theta(n,r)$ as stated in the theorem.
For this choice of $r$, the lower bound $\mathfrak S(N)\gg1$
follows from \cite[p.~2]{Vaughan-Wooley-2013}.
This completes the proof of Theorem~\ref{TeorAsForWaringPPSl}.

\end{document}